\documentclass{amsart}

\usepackage[utf8]{inputenc}
\usepackage[T1]{fontenc}
\usepackage{verbatim}
\usepackage{graphicx}
\usepackage{graphicx,caption2,psfrag,float,color}
\usepackage{amssymb}
\usepackage{amscd}
\usepackage{amsmath}
\usepackage{mathtools}

\usepackage[T1]{fontenc}

\newtheorem{theorem}{Theorem}[section]

\newtheorem{cor}[theorem]{Corollary}
\newtheorem{lemma}[theorem]{Lemma}

\numberwithin{equation}{subsection}
\newtheorem{definition}[theorem]{Definition}

\title{Recent results on large gaps between primes}
\author{Michael Th. Rassias}
\date{\today}
\address{Department of Mathematics and Engineering Sciences, Hellenic Military Academy, 16673 Vari Attikis, Greece }
\email{mthrassias@yahoo.com}
\thanks{}

\begin{document}

 \maketitle
 
\begin{abstract}
One of the themes of this paper is recent results on large gaps between primes. The first of these results has been achieved in the paper \cite{ford} by Ford, Green, Konyagin and Tao. It was later improved in the joint paper \cite{ford2} of these four authors with Maynard. One of the main ingredients of these results are old methods due to Erd\H{o}s and Rankin. Other ingredients are important breakthrough results due to Goldston, Pintz and Yildirim \cite{GPYI, GPYII, GPYIII}, and their extension by Maynard on small gaps between primes. All these previous results are discussed shortly. The results on the appearance of $k$-th powers of primes contained in those large gaps, obtained by the author in joint work with Maier \cite{MR, MR_Beatty, MR_PS} are based on a combination of the results just described with the matrix method of Maier.

\end{abstract}


\section{Introduction}

Let $p_n$ denote the $n$-th prime number, $d_n=p_{n+1}-p_n$. The topic of this  article are recent results on large values of $d_n$.\\
We start with a short overview on the historical development of the subject. The prime number theorem easily implies that $\frac{d_n}{\log p_n}$ is one on average:
$$\lim_{x\rightarrow \infty}\frac{1}{x}\sum_{n\leq x} \frac{d_n}{\log p_n}=1\:.$$
For an infinite sequence of values of $n$ this average value is superseded by a factor tending to infinity for $n\rightarrow \infty$. Let
$${G}(x):=\max_{p_{n+1}\leq x}(p_{n+1}-p_n)\:.$$
In 1931, Westzynthius \cite{Westzynthius}, improving on prior results of Backlund \cite{Backlund} and Brauer-Zeitz \cite{Brauer}, proved that 
$${G}(X) \gg \frac{\log X \log_2 X}{\log_3 X}.$$
(Here and in the sequel we define $\log_k x$ by $\log_1:=\log x$ and $\log_k:=\log(\log_{k-1} x)$, $(k>1)$).\\
In 1935, Erd\H{o}s \cite{erdos} sharpened this to 
$${G}(X)\gg \frac{\log X \log_2 X}{(\log_3 X)^2}$$
and in 1938 Rankin \cite{rankin} made a subsequent improvement
\[
{G}(X) \geq (c+o(1))\: \frac{\log X \log_2 X \log_4 X}{(\log_3 X)^2}\:. \tag{1.1}
\]
The constant $c$ was improved on several times (cf. \cite{maier_pom}, \cite{pintz}, \cite{rankin2}, \cite{Schonhage}). It was a famous prize problem of Erd\H{o}s to improve on the order of magnitude of the lower bound in (1.1). This problem was solved recently. In two independent papers, the paper \cite{ford} by the four authors K. Ford, B. Green, S. Konyagin and T. Tao, and the paper \cite{maynard2} by J. Maynard, it was shown that the constant in (1.1) could be taken to be arbitrarily large.\\
The methods of the two papers differed in some key aspects. The arguments in \cite{ford} used recent results from the papers \cite{green-tao}, \cite{green-tao2} by Green and Tao and the paper \cite{green-tao3} by Green, Tao and Ziegler on the number of solutions to linear equations in primes. 
The arguments in \cite{maynard2} by J. Maynard instead relied on multidimensional sieves introduced in \cite{maynard}, which in turn heavily relied on the breakthrough results of D. A. Goldston, J. Pintz and C. Y. Yildirim (cf. \cite{GPYI}, \cite{GPYII}, \cite{GPYIII}). \\
In this  article we shall restrict our description to the approach of the paper \cite{ford2} by all five authors. (We follow the notation the  version in \cite{ford2}, since this is also used in Maier-Rassias \cite{MR}).\\
Later on, the author of the present paper in a joint paper \cite{MR} with Maier obtained large gaps of the order of that in \cite{ford2} that contain a perfect $K$-th power of a prime for a fixed natural number $K\geq 2$. They combined the results and the methods of the paper \cite{ford2}, the method of the paper \cite{konyagin} of Ford, Heath-Brown and Konyagin with the Maier matrix method. The bulk of this  paper will deal with the description of the result of the paper \cite{ford2} and its $K$-version, the paper \cite{MR}.\\
The paper will be concluded with results about large gaps containing $K$-th powers of primes of special types: Beatty primes and Piatetski-Shapiro primes. 

\section{Short history of large gap results}

Starting with the papers \cite{erdos}-\cite{erdos2} of Erd\H{o}s, all the results on large gaps between primes are based on modifications of the Erd\H{o}s-Rankin method. Its basic features are as follows:\\
Let $x>1$. All steps are considered for $x\rightarrow \infty$. Let 
$$P(x):=\prod_{p<x} p\:,\ \ y>x\:.$$ 
By the prime number theorem we have  
$$P(x)=e^{x(1+o(1))}\:.$$
A system of congruence classes 
\begin{align*}
\{v\::\: v\equiv&\: h_{p_1} \bmod p_1\} \\
&\vdots \tag{2.1}\\
\{v\::\: v\equiv&\: h_{p_l} \bmod p_l\}\:,
\end{align*}
(with $p_1<\cdots < p_l$ being the primes less than $x$) is constructed, such that the congruence classes $h_{p_l} \bmod p_l$ cover the interval $(0, y]$.\\
Associated to the system (2.1) is the system of congruences 
\begin{align*}
\{m\equiv&\: -h_{p_1} \bmod p_1\} \\
&\vdots \tag{2.2}\\
\{m\equiv&\: -h_{p_l} \bmod p_l\}\:.
\end{align*}
By the Chinese Remainder Theorem the system (2.2) 
$$1\leq m< P(x)\:.$$
has a unique solution $m_0\bmod\: P(x)$\\
Let $u\in \mathbb{N}$, $1\leq u\leq y$. Then there is a $j$, $1\leq j\leq l$, such that 
$$u \equiv\: h_{p_j} \bmod p_j\:.$$
From (2.1) and (2.2)
$$m_0+u\equiv 0 \bmod p_j\:.$$
If $m_0$ is sufficiently large, then all integers $w\in (m_0, m_0+y]$ are composite. If 
$$p_n=\max\{ p \ \text{prime}\::\: p\leq m_0\}$$
then it follows that $p_{n+1}-p_n\geq y$, a large gap result.\\
The large gap problem has thus been reduced to a covering problem: Find a system of congruence classes that cover the interval $(0, y]$, where $y$ is as large as possible.\\
In all papers since Erd\H{o}s \cite{erdos, erdos2}, the covering system (2.1) has been constructed by a sequence of sieving steps. The set
$$\mathcal{S}(x):=\{p_1,\ldots, p_l\}$$
is partitioned into a disjoint union of subsets:
\[
 \mathcal{S}(x)=  \mathcal{S}_1\cup\cdots \cup\mathcal{S}_g\:.\tag{2.3}
\]
Associated to each sieving step $\mathfrak{st}_j$ is a choice of congruence-classes $f_p \bmod p$ for $p\in \mathcal{S}_j$. We also consider the sequence $\mathcal{R}_j$ of residual sets. It is recursively defined as follows:\\
The $0$-th residual set $\mathcal{R}_0$ covers the entire interval $(0, y]$. Thus
$$\mathcal{R}_0=\{n\in \mathbb{N}\cap (0, y]\}\:.$$
The $(j+1)^{st}$ residual set $\mathcal{R}_{j+1}$ is obtained by removing from $\mathcal{R}_j$ all the integers from $(0, y]$ congruent to $f_p \bmod p$ for some $p\in S_j$.
The sequence $(\mathfrak{st}_j)$, $1\leq j\leq g$ is complete, if $\mathcal{R}_{g-1}\neq \emptyset$, $\mathcal{R}_g=\emptyset$, that means all integers in $(0, y]$ have been removed. For a complete sequence of sieving steps the union 
$$\bigcup_{1\leq j\leq g}\:\bigcup_{p\in \mathcal{S}_j} (f_p\bmod p)$$ 
thus covers all of $(0, y]$ and the choice $h_p=f_p (p\in \mathcal{S}(x))$ in (2.1) gives a covering system of the desired kind.\\
In all versions of the Erd\H{o}s-Rankin method, the first sieving steps have been very similar.\\
We describe -- with minor modifications, adjusting to our notations -- the construction of the covering system (2.1) in Erd\H{o}s \cite{erdos, erdos2}.\\
One sets
$$X=\log x,\ y= \delta x \frac{\log x \log_3 x}{\log_2^2 x},\ Z=\frac{1}{2}\: x,$$
$$Y=\exp\left(\alpha\:  \frac{\log x \log_3 x}{\log_2 x}\right)=\exp\left(\alpha\:  \frac{\log x \log_3 x}{\log_2 x}\:(1+o(1)\right)\:.$$
The sets $\mathcal{S}_j$ of primes are defined as follows:
$$\mathcal{S}_1:=\{p\::\: p\in (0, X]\},\ \mathcal{S}_2:=\{p\::\: p\in (y, Z]\}\:.$$
$$\mathcal{S}_3:=\{p\::\: p\in (X, y]\},\ \mathcal{S}_4:=\{p\::\: p\in (Z, x]\},\ \mathcal{S}_5:=\{p\::\: p\in (x, y]\} \:.$$
For the first two sieving steps one defines the congruence classes 
$$h_{p_i} \bmod p_i,\ p_i\in \mathcal{S}_i\ (i=1, 2)$$
by 
$$h_{p_1}\equiv 0 \bmod p_1,\ h_{p_2}\equiv 0 \bmod p_2\:.$$
A simple consideration shows that for the second residual set $\mathcal{R}_2$ the intersection $\mathcal{R}_2\cap (x, y]$ is the union of a set $Q$ of prime numbers
$$Q:=\{ q\ \text{prime}\::\: x< q\leq y\}$$
with a set of $Z$-smooth integers, i.e. integers whose largest prime factor is $\leq Z$. A crucial fact in all variants of the Erd\H{o}s-Rankin method is, that the number of smooth integers is very small. This fact has been established by Rankin \cite{rankin} and Bruijn \cite{bruijn}.\\
A central idea of Rankin's method is $``$Rankin's trick$"$. Let us write $p^+(m)$ for the largest prime factor of $m$. Let $\Sigma'$ mean summation over all integers $n$ with $p^+(n)<y$. Then one has for $\eta>1$:
\begin{align*}
\sum_{\substack{m\leq x \\ p^+(m)\leq y}} 1&\leq \sum_{m\leq x} {\vphantom{\sum}}' 1 \leq  \sum_{m\leq x} {\vphantom{\sum}}' \left(\frac{x}{m}\right)^\eta\\
&=x^\eta \sum_{m\leq x} {\vphantom{\sum}}' \frac{1}{m^\eta}\leq x^\eta \prod_{p\leq y} \left(1-\frac{1}{p^\eta}\right)\:.
\end{align*}
The bound needed  follows by evaluating the product by the prime number theorem and by choosing $\eta$ optimally.\\
Thus, the elements of the second residual set essentially only consists of prime numbers, the number of $Z$-smooth numbers of the second residual set being negligible.\\
In the third sieving step in Erd\H{o}s \cite{erdos} the classes 
$$h_{p_3}\bmod p_3\ (p_3\in S_3)$$
are chosen by a greedy algorithm. In each step the congruence class not belonging to the previous congruence classes that contains the most elements of the residual set $\mathcal{R}_3$ is removed.\\
In each version of the Erd\H{o}s-Rankin method, there is a \textit{weak} sieving step, which we will not number, since this number might be different in different versions. Instead, we call it the \textit{weak sieving step}, since only few elements of the residual set are removed.\\
In the first paper \cite{erdos} of Erd\H{o}s, which is being discussed right now, in the fourth sieving step $\mathfrak{st}_4$ one uses the primes
$$p\in (0, x]\setminus (S_1\cup S_2 \cup S_3)$$
to remove the elements from the set $\mathcal{R}_4$.\\
An important quantity is the hitting number of the weak sieving step $\mathfrak{st}_{w_0}$. The hitting number of the prime $p\in \mathcal{S}_{w_0}$ is defined as the number of elements belonging to the congruence class $h_p \bmod p$. In all papers prior to \cite{maier_pom} this hitting number was bounded below by 1. Thus for each element $u$ of the residual set $\mathcal{R}_{w_0}$ a prime $p(u)\in \mathcal{S}_{w_0}$ could be found such that 
$$u\equiv h_p(u) \bmod p(u)$$
and thus the removal of a single element from the congruence class $(h_p(u) \bmod p(u))$ could be guaranteed. The progress in the papers was achieved not by changing the estimate for the hitting number, but by better estimates for the number of smooth integers.\\
In the paper \cite{maier_pom} by Maier-Pomerance, the hitting number in the weak sieving step for a positive proportion of the primes $p\in \mathcal{S}_{w_0}$ the hitting number was at least 2.\\
A further improvement was obtained in the paper \cite{pintz} by Pintz, where the hitting number was at least 2 for almost all primes in $\mathcal{S}_w$. We give a short sketch of these two papers.\\
The paper \cite{maier_pom} consists of an arithmetic part and a graph-theoretic part, combined with a modification of the Erd\H{o}s-Rankin method. The arithmetic information needed concerns the distribution of generalized twin primes in arithmetic progressions on average.\\
We recall definitions and theorems from \cite{maier_pom}. Fix some arbitrary, positive numbers $A, B$. For a given large number $N$, let $x_1, x_2$ satisfy
$$\frac{N}{(\log N)^A} \leq x_1< x_2 \leq N,\ x_2-x_1\geq \frac{N}{(\log N)^B}.$$
If $n$ is a positive integer, let
$$\bold{T}(n)=\{p\ \text{prime}\::\: x_1\leq p\leq x_2-n   \}$$
where as usual $p$ denotes a prime.\\
Further, if $l, M$ are positive integers, let
$$\bold{T}(n, l, M)=\{p\in \bold{T}(n)\::\: p\equiv l\bmod M\}.$$
Let 
$${T}(n,l,M)=|\bold{T}(n, l, M)|$$
and let
$$T(n)=\sum_{x_1<k\leq x_2-n} \frac{1}{\log k \log(k+n)}.$$
Let 
$$\alpha_0=2\prod_{p>2} \frac{p(p-2)}{(p-1)^2}=1,3203.$$
Then one has with a fixed constant $c_1>0$:
\begin{align*}
\sum_{\substack{n\leq x \\ n\equiv 0 \bmod 2}}\sum_{M\leq x^{c_1}} \max_{\substack{l\\ (l,m)=(n+l, M)=1}}&\max\left| T(x, n, l, M)-\frac{\alpha_0 T(x,n)}{\phi(M)}  \prod_{\substack{p\mid M \\ p>2}}\frac{p-1}{p-2} \right|\\
& \ll_{E} x^2(\log x)^{-E}. \tag{2.4}
\end{align*}
The result (2.4) is proven by application of the Hardy-Littlewood Circle method. We now come to the graph-theoretic part:\\
We have the following definitions:
\begin{definition}(\cite{maier_pom}, Definition 4.1')\\
Say that a graph $G$ is $N$-colored, if there is a function $\chi$ from the edge set of $G$ to $\{1, \ldots, N\}$.
\end{definition}
In the paper \cite{maier_pom}, first a graph is discussed, whose properties are idealized and thus simpler to formulate than the properties really needed for the applications. A proof of the existence of certain colored subgraphs (partial matchings) is given. Then the graphs with properties needed for the applications are discussed. The existence of certain colored subgraphs is given without proof. The proof can easily be obtained by a modification of the proof for the idealized graphs mentioned above. For the sketch of the details we cite (\cite{maier_pom}, Definition 4.2). \\
Say an $N$-colored graph $G$ is $K$-uniform, if $K\mid N$ and there are integers $S, T$ such that\\
(i) Each color in $\{1,\ldots, N\}$ is assigned to exactly $S$ edges of $G$.\\
(ii) For each $i=1, \ldots, K$ and each vertex $V$ in $G$, there are exactly $T/K$ edges $E$ coincident at $V$ with color in $((i-1)N/K, iN/k)$. Thus each vertex of $G$ has valence $T$.\\
One has 
\begin{theorem}(\cite{maier_pom}, Theorem 4.1)\\
Say $G$ is a $K$-uniform, $N$-colored graph with $N$ vertices, where $c\geq 1$. Then there is a set of $B$ mutually non-coincident edges with distinct colors such that
$$B>\frac{cN}{4}\left(1-\exp\left(-\frac{4}{c}+\frac{8}{c^2K}\right)\right)$$
\end{theorem}
We describe the construction of these edges:\\
Let $S, T$ be as in (\cite{maier_pom}, Definition 4.2).\\
Let $B$, be the largest collection of mutually noncoincident edges with distinct colors in $(0, N/K]$. After $B_1, \ldots, B_{i-1}$ have been chosen and $i\leq K$, let $B_i$ be the largest collection of edges of $G$ with distinct colors in $((i-1)N/K, iN/K)$ such that the members of
$$B_1\cup \cdots \cup B_i$$ 
are mutually noncoincident. Let $\beta_i$ be such that $|B_i|=\beta_i N$ and let
$$\beta=\beta_1+\cdots+\beta_K.$$
It can be shown that 
$$\beta>\frac{c}{4}\left(1-\exp\left(-\frac{4}{c}+\frac{8}{c^2K}\right)\right).$$
We now describe the modifications suited for applications.
\begin{definition}(\cite{maier_pom}, Definition 4.2')\\
Let $K$ be a positive integer and let $C>0$, $\delta\geq 0$ be arbitrary. Say an $N$-colored graph $G$ with $N$ vertices is $(K, C, \delta)$-uniform, if there are numbers $S, T$ such that \\
(i) but for at most $\delta N$ exceptions, each color in $\{1, \ldots, N\}$ is assigned to between $(1-\delta)\mathcal{S}$ and $(1+\delta)\mathcal{S}$ edges of $G$,\\
(ii) if we let $n(V, i)$ denoted the number of edges coincident at the vertex $V$ with color in $((i-1)N/K, iN/K)$, then 
$$n(V, i)\leq CT/K$$ 
for each $i=1, \ldots, K$, but for at most $\mathcal{S} M$ exceptional vertices $V$, we have 
$$(1-\delta)T/K \leq n(V, i)\leq (1+\delta)T/K$$
for each $i=1, \ldots, K$.
\end{definition}
Then we have the following result:
\begin{theorem}(\cite{maier_pom}, Theorem 4.1')\label{rthm24}\\
Let $C>0$, $\eta>0$ be arbitrary. There is a number $K(C, \eta)$ such that for each integer $K\geq K(C, \eta)$ there is some 
$\delta=\delta(C, \eta, K)>0$ with the property that each $(K, C, \delta)$-uniform, $N$-colored graph with $cN$ vertices, where $c\geq 1$, has a set of $B$ mutually noncoincident edges with distinct  colors, where
$$B>(1-\eta)\frac{cN}{4}\left(1-\exp\left(-\frac{4}{C}\right)\right)\:.$$
\end{theorem}

We now describe the application of the Erd\H{o}s-Rankin method in the paper \cite{maier_pom} and its combination with the arithmetic and graph-theoretic results just mentioned.\\
Let
\begin{align*}
y&:= c' e^\gamma x\log x\log\log\log x(\log\log x)^{-2}\\
z&:=x/\log\log x\\
v&:=\exp\{(1-\epsilon)\log x\log\log\log x(\log\log x)^{-2}.
\end{align*}
The first two sieving steps are as follows:\\
For the system of congruence classes $h_{p_1}\bmod p_1$ as described in (2.1) we choose:
\begin{align*}
&h_{p_1}=0\ \text{for every prime } p_1\in \mathcal{S}_1:=(y, z]\:,\\
&h_{p_2}=1 \ \text{for every prime } p_2\in \mathcal{S}_2:=(1, y]\:.
\end{align*}
The first residual set $\mathcal{R}_1$ is the disjoint union $\mathcal{R}_{(1)}\cup \mathcal{R}_{(2)}$, where $\mathcal{R}_{(1)}$ is the set of integers in $(1, y]$ divisible by some prime $p>z$ and $\mathcal{R}_{(z)}$ is the set of $v$-smooth integers in $(1, y]$. Let $\mathcal{R}$ be the members of the second residual set that are in $\mathcal{R}_{(1)}$ and let $\mathcal{R}'$ be the members of the second residual set that are in $\mathcal{R}_{(2)}$. Then
$$\mathcal{R}=\bigcup_{m\leq Y/z} \mathcal{R}_m,$$
where
\begin{align*}
&\mathcal{R}_m:=\{ mp\::\: z< p\leq U/m,\ (mp-1, P(y))=1\}\:,\\
&\mathcal{R}':=\{ n\leq y\::\: p\mid n\ \Rightarrow\ p\leq y, q\mid n-1\ \Rightarrow\ q>y \}\:.
\end{align*}
It is again important that the number of smooth integers is small and it easily follows that 
$$|\mathcal{R}'|\ll \frac{x}{(\log x)^{1+\epsilon}}.$$
For the weak sieving step one now applies the graph-theoretic results (Theorem \ref{rthm24}).\\
One defines a graph whose vertex set is $\mathcal{R}_m$. Let
$$k_0:=\prod_{r<\log\log\log x} r\:.$$
Define 
$$r_m=\frac{\alpha_0}{m\log\log x}\:\prod_{r\mid m}\frac{r-1}{r-2}\:.$$
Let $Q_m$ denote the set of primes $q$ in the interval
$$\left(\left(1-\sum_{j=1}^m r_j\right),\ \left(1-\sum_{j=1}^{m-1} r_j\right) x \right]\:.$$
Let $Q_m$ be the graph with vertex set $\mathcal{R}_m$ and such that $mp, mp'\in \mathcal{R}_m$ are connected by an edge if and only if  $|p'-p|=k_0q$ for some $q\in Q_m$.\\
Define the $``$color$"$ of an edge by the prime $q$, so that $G_m$ is a $|Q_m|$-colored graph. From the arithmetic information, combined with standard sieves, it can easily be deduced that the graphs $G_m$ satisfy the conditions of the graph-theoretic result (\cite{maier_pom}, Definition 4.2). Thus the graphs $G_m$ contain a sufficient number of edges $(mp, mp')$ and thus pairs $(p, p')$ with
$$p\equiv p' (\bmod\: q).$$
We consider the system
\begin{align*}
\{v\::\: v&\equiv h_{p_1}\bmod\: p_1 \} \:,\\
&\vdots  \tag{2.1}\\
\{v\::\: v&\equiv h_{p_l}\bmod\: p_l \}\:,
\end{align*}
for $p_j=q\in Q_m.$\\
If we determine $$h_{p_j}\bmod\: p_j = h_q \bmod\: q$$ by 
$$h_q\equiv p\equiv p'\bmod \: q$$
then the hitting number for the prime $p_j=q$ is 2. Thus by weak sieving step, two members of the residual set are removed for each prime $q$. The weak sieving step is completed by removing one member of the residual set for the remaining primes.\\
The paper \cite{pintz} by Pintz, contains exactly the same arithmetic information as the paper \cite{maier_pom} by Maier and Pomerance, whereas the graph-theoretic construction is different. The edges of the graphs are obtained by a random construction and a hitting number of 2 for almost all primes in the weak sieving step is achieved. \\
The order of magnitude of $G(X)$ could finally be improved in the paper \cite{ford}. The result is:
$$G(x)\geq f(x)\:\frac{\log X\log_2 X\log_4 X}{(\log_3 X)^2}\:,$$
with $f(X)\rightarrow \infty$ for $X\rightarrow \infty$.\\
The paper is related to the work on long arithmetic progressions consisting of primes by Green and Tao  \cite{green-tao}, \cite{green-tao2} and work by Green, Tao and Ziegler \cite{green-tao3} on linear equations in primes. The authors manage to remove long arithmetic progressions of primes in the weak sieving step and thus are able to obtain a hitting number tending to infinity with $X$. We shall not describe any more details of this paper. Simultaneously and independently, James Maynard \cite{maynard} achieved progress based on multidimentional sieve methods. The authors of the paper \cite{ford} and Maynard in \cite{ford2} joined their efforts to prove
$$G(X)\geq C\log X\log_2 X\log_4 X (\log_3 X)^{-1}$$
for a constant $C>0$.\\
Again the hitting number in the weak sieving step tends to infinity for $x\rightarrow \infty$. Whereas in the papers \cite{maier_pom}, \cite{pintz} by Maier-Pomerance and Pintz, the pairs of the integers removed in the weak sieving step were interpreted as edges of a graph, now the tuplets of integers removed are seen as edges of a hypergraph. One uses a hypergraph covering theorem generalizing a result of Pippenger and Spencer \cite{pip} using the R\"odl nibble method \cite{rodl}.\\
The choice of sieve weights is related to the great breakthrough results on small gaps between consecutive primes, based on the Goldston-Pintz-Yildirim (GPY) sieve and Maynard's improvement of it. We give a short overview.

\section{Small gaps, GPY-Sieve and Maynard's improvement}
The first non-trivial bound was proved by Erd\H{o}s \cite{erdos, erdos2} who showed that
$$\liminf_{n\rightarrow \infty} \frac{d_n}{\log p_n} <1\:.$$
By the application of Selberg's sieve he showed that pairs of primes $(\tilde{p_1}, \tilde{p_2})$ with a fixed difference cannot appear too often.\\
The first major breakthrough was achieved by Bombieri-Davenport \cite{bombieri}, who could show that 
$$\liminf_{n\rightarrow \infty} \frac{d_n}{\log p_n} \leq 0.46650\ldots$$
Let 
\[
Z(2n):=\sum_{\substack{p, p'\leq x\\ p'-p=2n}} (\log p)(\log p')\:, \tag{3.1}
\]
\[
\mathcal{S}(\alpha):=\sum_{p \leq x}(\log p)e(p\alpha)  \tag{3.2}
\]
$$\mathcal{U}(\alpha):=\sum_{m=-L}^L e(2m\alpha)\ (e(\alpha)=e^{2\pi i \alpha},\ L\in\mathbb{N}).$$
Then 
$$T(\alpha):=|\mathcal{U}(\alpha)|^2=\sum_{j=-2L}^{2L} t(j)e(2j\alpha)\:,$$
with $t(j):=2L+1-|j|.$\\
One row considers the integral
\[
I(x)=\int_0^1 |\mathcal{S}(\alpha)|^2 T(\alpha) d\alpha\:. \tag{3.3}
\]
By orthogonality one obtains:
\[
I(x)=t(0)Z(0)+2\sum_{m=1}^{2L} t(m) Z(2m)\:.  \tag{3.4}
\]
One now tries to establish a lower bound for I(x). This bound can be combined with upper bounds for $Z(2m)$ for large values of $m$ to obtain estimates $Z(2m)>0$ for small values of $m$. Thus gaps of size $2m$ exist.

$$ $$
$$ $$
$$ $$
$$ $$

These estimates became possible by application of the Bombieri-Vinogradov theorem, proven one year before \cite{daven_2000}.\\ 
For its formulation the following definition will be useful:
\begin{definition}\label{defnr31}
Let 
$$\theta(x, q, a)= \sum_{\substack{p\leq x\\ p\equiv a (\bmod\: q)}}\log p\:.$$
We say that the primes have an admissible level of distribution $\theta$ if 
$$\sum_{q\leq x^{\delta-\epsilon}} \max_{(a,q)=1}\left| \theta(x; q, a)-\frac{x}{\phi(q)} \right| \ll \frac{x}{(\log x)^A}$$
holds for any $A>0$ and any $\epsilon>0$.
\end{definition}

The Bombieri-Vinogradov theorem now states that:\\ 
For any $A>0$, there is a $B=B(A)$ such that, for 
$$Q=x^{1/2}(\log x)^{-B}\::$$
\[
\sum_{q\leq Q} \max_{(a,q)=1}\left| \theta(x; q, a)-\frac{x}{\phi(q)} \right| \ll \frac{x}{(\log x)^A}\:.  \tag{3.4}
\]
This implies that the primes have admissible level of distribution $1/2$.

\begin{definition}\label{defn41}
We say that the primes have an \textbf{admissible level of ditribution} ${\vartheta}$, if (3.4) holds for any $A>0$ and any $\epsilon>0$ with $Q=x^{\vartheta-\epsilon}$.
\end{definition}
A great breakthrough was achieved in the paper \cite{GPYI}. They consider admissible $k$-tuples for which we reproduce the definition:

\begin{definition}\label{defn42}
$\mathcal{U}$ is called admissible, if for each prime $p$ the number $v_p(\mathcal{U})$ of distinct residue classes modulo $p$ occupied by elements of $\mathcal{U}$, satisfies $v_p(\mathcal{U})<p$.
\end{definition}

The two main results in the paper \cite{GPYI} of Goldston, Pintz and Yildirim are

\begin{theorem}\label{thm1} (\cite{GPYI}, Theorem 3.3)
\begin{align*}
&\text{Suppose the primes have level of distribution $\vartheta>1/2$. Then there exists}\\ 
&\text{an explicitly calculable constant $C(\vartheta)$ depending only on $\vartheta$ such that any}\\ 
&\text{admissible $k$-tuple with $k\geq C(\vartheta)$ contains at least two primes infinitely}\\
&\text{often. Specifically if $\vartheta\geq 0.971$ then this is true for $k\geq 6$.} 
\end{align*}
\end{theorem}

\begin{theorem}\label{thm1} (\cite{GPYI}, Theorem 3.4)\\
We have
\[
\Delta_1:=\liminf_{n\rightarrow \infty} \frac{p_{n+1}-p_n}{\log p_n}=0\:.
\]
\end{theorem}
The method of Goldston, Pintz, Yildirim has also become known as the GPY-sieve.\\
There are several overview articles on the history of the GPY-method (cf. \cite{Broughan, GPYIII}).\\
The overview article most relevant for this paper is due to Maynard \cite{maynard0}, whose improvements of the GPY-sieve is of crucial importance for the large gap results described in this paper.\\
Before we recall Maynard's description, we should mention another milestone which however is not relevant for large gap results. The result of Yitang Zhang \cite{zhang} from 2014. He proves the existence  of infinitely many bounded gaps. He does not establish an admissible level of distribution $\vartheta>1/2$, which by (4.5) would imply the result, but succeeds to replace the sum
$$\sum_{q\leq Q}\max_{(a,q)=1}\left|\theta(x,q,a)-\frac{x}{\phi(q)}  \right|$$
by a sum over smooth moduli.\\
We now come to the short description of the GPY-method and its improvement by Maynard, closely following the papers ``Small gaps between primes$"$ by Maynard \cite{maynard0}. One of the main results of \cite{maynard0} is:
\begin{theorem}\label{thm11}(of \cite{maynard0})\\
Let $m\in\mathbb{N}$. We have
$$\liminf_{n\rightarrow \infty} (p_{n+m}-p_n)\ll m^3 e^{4m}\:.$$
\end{theorem}
Tao (in private communication to Maynard) has independently proven Theorem \ref{thm11} (with a slightly weaker bound at much the same time).\\
Theorem \ref{thm11} implies that for every $H>0$ there exist intervals whose lengths depend only on $H$ with arbitrarily large initial point, that contain at least $H$ primes.\\
Now we follow \cite{maynard0} for a short description of the GPY-sieve and its improvement.\\
Let $\mathcal{U}:=\{h_1, \ldots, h_k\}$ be an admissible $k$-tuple. One considers the sum
\[
S(N, \rho)=\sum_{N\leq n\leq 2N}\left(\sum_{i=1}^k \chi_{\mathbb{P}}(n+h_i)-\rho\right)w_n\:. 
\]
Here $\chi_{\mathbb{P}}$ is the characteristic function of the primes, $\rho>0$ and $w_n$ are non-negative weights. If one can show that $S(N,\rho)>0$, then at least one term in the sum over $n$ must have a positive contribution. By the non-negativity of $w_n$, this means that there must be some integer $n\in [N, 2N]$ such that at least $\lfloor \rho+1\rfloor$ of the $n+h_i$, are prime.\\
The weights $w_n$ are typically  chosen  to mimic Selberg sieve weights. Estimating (4.6) can be interpreted as a $``$$k$-dimensional$"$ sieve problem. The standard Selberg $k$-dimensional weights are
$$w_n:=\left(\sum_{\substack{d\mid \prod_{i=1}^k (n+h_i)\\ d<R}}\lambda_d\right),\ \ \lambda_d:=\mu(d)\left(\log \frac{R}{d}\right)^k\:.$$
The key new idea in the paper \cite{GPYI} of Goldston, Pintz, Yildirim was to consider more general sieve weights of the form
$$\lambda_d:=\mu(d)F\left(\log \frac{R}{d}\right)$$
for a suitable smooth function $F$.\\
Goldston, Pintz and Yildirim chose $F(x):=x^{k+l}$ for suitable $l\in\mathbb{N}$, which has been shown to be essentially optimal, when $k$ is large.\\
The new ingredient in Maynard's method is to consider a more general form of the sieve weights
$$w_n=\left(\sum_{d_i\mid n+h_i,\forall i} \lambda_{d_1\ldots d_k}\right)^2\:.$$
The results of \cite{maynard0} have been modified and extended in the paper \cite{maynard} $``$Dense clusters of primes in subsets$"$ of Maynard. Some of his results and their applications will be described later in this paper.

\section{Large gaps with improved order of magnitude and its $K$-version, Part I}

Here we state the theorems from \cite{ford2} and \cite{MR} and sketch their proofs.\\
We number definitions and theorems in the following manner:\\
Definition (resp. Theorem) $X$ of paper $[Y]$ (in the list of references) are referred to as ($[Y]$, Definition (resp. Theorem) $X$).\\
We start with a list of the theorems from \cite{ford2} and the definitions relevant for them:

\begin{theorem}\label{thm41} (\cite{ford2}, Theorem 1)(Large prime gaps)\\
For any sufficiently large $X$, one has
\[
G(X)\gg \frac{\log X \log_2 X \log_4 X}{\log_3 X}  \tag{4.1}
\]
for sufficiently large $X$. The implied constant is effective. 
\end{theorem}
\begin{definition}(\cite{ford2}, Definition (3.1))\\
$$y:= cx\:\frac{\log x\log_3 x}{\log_2 x}\:,$$
where $c$ is a certain (small) fixed positive constant.
\end{definition}
\begin{definition}(\cite{ford2}, Definition (3.2))\\
$$z:=x^{\log_3 x/(4\log_2 x)}.$$
\end{definition}
\begin{definition}(\cite{ford2}, Definitions (3.3)-(3.5))\\
\begin{align*}
S&:=\{s\ \text{prime}\::\: \log^{20} x< s\leq z\}\\
P&:=\{p\ \text{prime}\::\: x/2<p\leq x\}\\
Q&:=\{q\ \text{prime}\::\: x<q\leq y\}.
\end{align*}
\end{definition}
For congruence classes
\[
\vec{a}:=(a_s \bmod s)_{s\in S} 
\]
and
\[
\vec{b}:=(b_p \bmod p)_{p\in P} 
\]
define the sifted sets
\[
S(\vec{a}):=\{n\in\mathbb{Z}\::\: n\not\equiv a_s (\bmod s)\ \text{for all $s\in S$}\} 
\]
and likewise
\[
S(\vec{b}):=\{n\in\mathbb{Z}\::\: n\not\equiv b_p (\bmod p)\ \text{for all $p\in P$}\} 
\]
\begin{theorem}\label{thm42}(\cite{ford2}, Theorem 2  - Sieving primes)\\
Let $x$ be sufficiently large and suppose that $y$ obeys (3.1). Then there are vectors 
$$\vec{a}=(a_s \bmod s)_{s\in S}\ \ \text{and}\ \ \vec{b}=(b_p \bmod p)_{p\in P},$$
such that 
\[
\#(Q\cap S(\vec{a})\cap S(\vec{b}))\ll \frac{x}{\log x}\:. 
\]
\end{theorem}
\begin{theorem}\label{thm43}(\cite{ford2}, Theorem 3)(Probabilistic covering)\\
There exists a constant $C_0\geq 1$ such that the following holds. Let $D, r, A$, \mbox{$0<\kappa\leq 1/2$,} and let $m\geq 0$ be an integer. Let $\delta>0$ satisfy the smallness bound
\[
\delta\leq \left(\frac{\kappa^A}{C_0\exp(AD)}  \right)^{10^{m+2}} 
\]
\end{theorem}
Let $I_1, \ldots, I_m$ be disjoint finite non-empty sets and let $V$ be a finite set. For each $1\leq j\leq m$ and $i\in I_j$, let $\bold{e}_i$ be a random finite subset of $V$. Assume the following:\\
$\bullet$ (Edges not too large) With probability 1, we have for all $j=1,\ldots,m$ and $i\in I_j$
$$\# \bold{e}_i\leq r_i$$
$\bullet$ (Each sieving step is sparse) For all $j=1,\ldots,m$, $i\in I_j$ and $v\in V$,
\[
\mathbb{P}(v\in \bold{e}_i)\leq \frac{\delta}{(\# I_j)^{1/2}} 
\]
$\bullet$ (Very small codegrees) For every $j=1,\ldots, m$ and distinct $v_1, v_2\in V$,
\[
\sum_{i\in I_j} \mathbb{P}(v_1, v_2\in \bold{e}_i)\leq \delta 
\]
$\bullet$ (Degree bound) If for every $v\in V$ and $j=1,\ldots, m$ we introduce the normalized degrees
\[
d_{I_j}(v):=\sum_{i\in I_j}\mathbb{P}(v\in \bold{e}_i)
\]
and then recursively define the quantities $P_j(v)$ for $j=0,\ldots, m$ and $v\in V$ by setting
\[
P_0(v):=1 
\]
and
\[
P_{j+1}(v):= P_j(v)\exp(-d_{I_{j+1}}(v)/P_j(v)) 
\]
for $j=0, \ldots, m-1$ and $v\in V$, then we have 
\[
d_{I_j}(v)\leq DP_{j-1}(v),\ \ (1\leq j\leq m, v\in V)
\]
and
\[
P_j(v)\geq \kappa \ \ (0\leq j\leq m, v \in V)\:. 
\]
Then we can find random variables $e_i'$ for each $i\in \bigcup_{j=1}^m I_j$ with the following properties:\\
(a) For each $i\in \bigcup_{j=1}^m I_j$, the essential support of $\bold{e}_i'$ is contained in the essential support of $\bold{e}_i$, union the empty set singleton $\{\emptyset\}$. In other words, almost surely $\bold{e}_i'$ is either empty, or is a set that $\bold{e}_i$ also attains with positive probability.\\
(b) For any $0\leq J\leq m$ and any finite subset $e$ of $V$ with $\# e\leq A-2vJ$, one has
\[ 
\mathbb{P}\left( e\subseteq V\setminus \bigcup_{j=1}^J\bigcup_{i\in I_j} \bold{e}_i'\right)=\left(1+O_{\leq}\left(\delta^{1/10^{J+1}}\right) \right)P_j(e) 
\]
where
\[
P_j(e):=\prod_{v\in e} P_j(v)\:. 
\]
\begin{cor}\label{cor44} (\cite{ford2}, Corollary 4)\\
Let $x\rightarrow \infty$. Let $P', Q'$ be sets with $\# P'\leq x$ and $ (\log_2 x)^3< \# Q' \leq x^{100}$. For each $p\in P'$, let $\bold{e}_p$ be a random subset of $Q'$ satisfying the size bound:
\[
\# \bold{e}_p \leq r=O\left(\frac{\log x\log_3 x}{\log_2^2 x}\right)\ \ (p\in P') 
\]
Assume the following:\\
$\bullet$ (Sparcity) For all $p\in P'$ and $q\in Q'$
\[
\mathbb{P}(q\in \bold{e}_p)\leq x^{-1/2-1/10}\:.  
\]
$\bullet$ (Small codegrees) For any distinct $q_1, q_2\in Q'$
\[
\sum_{p\in P'} \mathbb{P}(q_1, q_2\in \bold{e}_p)\leq x^{-1/20}\:. 
\] 
$\bullet$ (Elements covered more than once in expectation) For all but at most $\frac{1}{(\log_2 x)^2}\# Q'$\  elements $q\in Q'$ we have:
\[
\sum_{q\in P'}\mathbb{P}(q\in \bold{e}_p)=C+O_{\leq}\left(\frac{1}{(\log_2 x)^2}\right)  
\]
for some quantity $C$, independent of $q$, satisfying
\[
\frac{5}{4}\log 5\leq C\leq 1\:.  
\]
Then for any positive integer $m$ with 
\[
m\leq \frac{\log_3 x}{\log 5}
\]
We can find random sets $\bold{e}_p'\subseteq Q'$ for each $p\in P'$ such that $\bold{e}_p'$ is either empty or a subset of $Q'$ which $\bold{e}_p$ attains with positive probability, and that 
$$\#\{q\in Q'\::\: q\not\in \bold{e}_p'\ \text{for all}\ p\in P'\}\sim 5^{-m}\# Q'$$
with probability $1-o(1)$. More generally, for any $Q''\subset Q'$ with cardinality at least \mbox{$(\# Q')/\sqrt{\log_2 x}$} one has
$$\#\{q\in Q''\::\: q\not\in \bold{e}_p'\ \text{for all}\ p\in P'\}\sim 5^{-m}\# Q''$$
with probability $1-o(1)$. The decay rates in the $o(1)$ and $\sim$ notation are uniform in $P', Q', Q''$.\\
\end{cor}

\begin{theorem}\label{thm44} (\cite{ford2}, Theorem 4)(Random construction)
Let $x$ be a sufficiently large real number and define $y$ by (3.1). Then there is a quantity $\mathcal{C}$ with
\[
\mathcal{C}\asymp \frac{1}{c} 
\]
with the implied constants independent of $c$, a tuple of positive integers $(h_1,\ldots, h_r)$ with $r\leq \sqrt{\log x}$, and some way to choose random vectors $\vec{\bold{a}}=(\bold{a}_s\ \bmod\: s)_{s\in S}$ and $\vec{\bold{n}}=(\bold{n}_p)_{p\in P}$ of congruence classes $\bold{a}_s \bmod s$ and integers $\bold{n}_p$ respectively, obeying the following:\\
$\bullet$ For every $\vec{a}$ in the essential range of $\vec{\bold{a}}$, one has
\[
\mathbb{P}(q\in \bold{e}_p(\vec{a})\ |\ \vec{\bold{a}}=\vec{a})\leq x^{1/2-1/20}\ \ (p\in P)\:,  
\]
where 
$$\bold{e}_p(\vec{a}):=\{\bold{n}_p+h_ip\::\: 1\leq i\leq r\}\cap Q\cap S(\vec{a})\:.$$
$\bullet$ With probability $1-o(1)$ we have that
\[
\#(Q\cap S(\vec{\bold{a}}))\sim 80\:c\:\frac{x}{\log x}\log_2 x.  
\]
$\bullet$  Call an element $\vec{a}$ in the essential range of $\vec{\bold{a}}$ good if, for all but at most $\frac{x}{\log x\log_2 x}$ elements $q\in Q\cap S(\vec{a})$, one has
\[
\sum_{p\in P}\mathbb{P}(q\in \bold{e}_p(\vec{a})\ |\ \vec{\bold{a}}=\vec{a})=\mathcal{C}+O_{\leq}\left(\frac{1}{(\log_2 x)^2} \right) 
\]
Then $\vec{\bold{a}}$ is good with probability $1-o(1)$.
\end{theorem}

\textit{Theorem and Definitions from \cite{MR}}

\begin{theorem}(\cite{MR}, Theorem 1.1)\\
There is a constant $c>0$ and infinitely many $n$, such that
$$p_{n+1}-p_n\geq c\: \frac{\log p_n \log_2 p_n \log_4 p_n}{\log_3 p_n}$$
and the interval $[p_n, p_{n+1}]$ contains the $K$-th power of a prime.
\end{theorem}
\begin{definition}(\cite{MR}, Definitions (3.1)-(3.5))\\
$$y:= cx\:\frac{\log x\log_3 x}{\log_2 x}\:,$$
where $c$ is a fixed positive constant. Let
$$z:=x^{\log_3 x/(4\log_2 x)}.$$
and introduce the three disjoint sets of primes
\begin{align*}
S&:=\{s\ \text{prime}\::\: \log^{20} x< s\leq z\}\\
P&:=\{p\ \text{prime}\::\: x/2<p\leq x\}\\
Q&:=\{q\ \text{prime}\::\: x<q\leq y\}.
\end{align*}
For residue classes
\[
\vec{a}:=(a_s \bmod s)_{s\in S} 
\]
and
\[
\vec{b}:=(b_p \bmod p)_{p\in P} 
\]
define the sifted sets
\[
S(\vec{a}):=\{n\in\mathbb{Z}\::\: n\not\equiv a_s (\bmod s)\ \text{for all $s\in S$}\} 
\]
and likewise
\[
S(\vec{b}):=\{n\in\mathbb{Z}\::\: n\not\equiv b_p (\bmod p)\ \text{for all $p\in P$}\}. 
\]
We set
\begin{align*}
\mathcal{A}_{(K)}&:=\{\vec{a}=(a_s \bmod s)_{s\in S}\::\: \exists\ c_s\ \text{such that}\\
&\ \  a_s\equiv 1-(c_s+1)^K(\bmod s), c_s\not\equiv -1 (\bmod s)\}  
\end{align*}
\begin{align*}
\mathcal{B}_{(K)}&:=\{\vec{b}=(b_p \bmod p)_{p\in P}\::\: \exists\ d_p\ \text{such that}\\
&\ \  b_p\equiv 1-(d_p+1)^K(\bmod p), b_p\not\equiv -1 (\bmod p)\} . 
\end{align*}
\end{definition}
\begin{theorem}(\cite{MR}, Theorem 3.1)(Sieving primes)\\
Let $x$ be sufficiently large and suppose that $y$ obeys (3.1). Then there are vectors $\vec{a}\in\mathcal{A}_{(K)}$ and $\vec{b}\in\mathcal{B}_{(K)}$, such that 
\[
\#(Q\cap S(\vec{a})\cap S(\vec{b}))\ll \frac{x}{\log x}\:. 
\]
\end{theorem}
\begin{theorem}(\cite{MR}, Theorem 4.1)\\
(Has wording identical to (\cite{ford2}, Theorem 3).
\end{theorem}
\begin{cor}(\cite{MR}, Corollary 4.2)\\
(Has wording identical to (\cite{ford2}, Corollary 3).
\end{cor}
\begin{theorem}(\cite{MR}, Theorem 4.3)(Random construction)\\
(Has wording identical to (\cite{ford2}, Theorem 4).
\end{theorem}
\begin{definition}(\cite{MR}, Definition 6.1)\\
An admissible $r$-tuple is a tuple $(h_1, \ldots, h_r)$ of distinct integers that do not cover all residue classes modulo $p$ for any prime $p$.\\
For $(u, K)=1$ we define
$$S_u:=\{s\::\: s\ \text{prime},\ s\equiv u\:(\bmod\: K),\ (\log x)^{20}< s\leq z\}$$
$$d(u)=(u-1, K),\ r^*(u)=\frac{1}{d(u)}\sum_{s\in S_u} s^{-1}.$$
For $n\in [x, y]$ let
$$r(n, u)=\sum_{\substack{s\in S_u\::\: \exists c_s\::\: n\equiv 1-(c_s+1)^K(\bmod s)\\ c_s\not\equiv -1 (\bmod s)}} s^{-1}\:.$$
We set
\begin{align*}
\mathcal{G}&=\{n\::\: n\in [x, y], |r(n,u)-r^*(u)|\leq (\log x)^{-1/40}\\ 
&\text{for all}\ u(\bmod K),\ (u, K)=1\}.
\end{align*}
For an admissible $r$-tuple to be specified later and for primes $p$ with  $x/2<p\leq x$ we set
$$\mathcal{G}(p)=\{n\in\mathcal{G}\::\: n+(h_i-h_l)p\in\mathcal{G}, \forall i, l\leq r\}.$$
\end{definition}
\begin{theorem}\label{thm45} (\cite{MR}, Theorem 6.2  - Existence of good sieve weights)\\
Let $x$ be a sufficiently large real number and let $y$ be any quantity obeying (3.1). Let $P, Q$ be defined by (3.2)-(3.5). Let $r$ be a positive integer with 
\[
r_0\leq r\leq \log^{\eta_0} x 
\]
for some sufficiently large absolute constant $r_0$ and some sufficinetly small $\eta_0>0$.\\ 
Let $(h_1, \ldots, h_r)$ be an admissible $r$-tuple contained in $[2r^2]$. Then one can find a positive quantity
\[
\tau \geq x^{-o(1)} 
\]
and a positive quantity $u=u(r)$ depending only on $r$
\[
u\asymp \log r 
\]
and a non-negative function $w\::\: P\times\mathbb{Z}\rightarrow\mathbb{R}^+$ supported on $P\times(\mathbb{Z}\cap [-y, y])$ with the following properties:\\
$\bullet$ $w(p, n)=0$ unless 
$$n\equiv 1-(d_p+1)^K (\bmod\: p)\:,\ \text{for some $d_p\in\mathbb{Z}$}$$
$$d_p\not\equiv -1\:(\bmod\:p)\ \text{and}\ n\in\mathcal{G}(p).$$
$\bullet$ Uniformly for every $p\in P$, one has
\[
\sum_{n\in\mathbb{Z}} w(p, n)=\left(1+O\left(\frac{1}{\log_2^{10} x}\right)\right)\tau\:\frac{y}{\log x} 
\]
$\bullet$ Uniformly for every $q\in Q$ and $i=1,\ldots r$ one has
\[
\sum_{p\in P} w(p, q-h_i p)=\left(1+O\left(\frac{1}{\log_2^{10} x}\right)\right)\tau\:\frac{u}{r}\:\frac{x}{2\log^r x} 
\]
$\bullet$ Uniformly for every $h=O(y/x)$ that is not equal to any of the $h_i$, one has
\[
\sum_{q\in Q}\sum_{p\in P} w(p, q-h_p)=O\left(\frac{1}{\log_2^{10} x}\:\tau\: \frac{x}{\log^r x}\:\frac{y}{\log\log x} \right) 
\]
uniformly for all $p\in P$ and $n\in\mathbb{Z}$.
\[
w(p, q)=O\left(x^{1/3+o(1)} \right).  
\]
\end{theorem}
\noindent In \cite{ford2} we have the following dependency graph for the proof of (\cite{ford2}, Theorem 1).

\[
\text{(\cite{ford2}, Thm. 5)}\ \Rightarrow\  \text{(\cite{ford2}, Thm. 4)}\ \Rightarrow\   \text{(\cite{ford2}, Thm. 2)}\ \Rightarrow\ \text{(\cite{ford2}, Thm. 1)}. \tag{4.2}
\]
Replacing these theorems by their $K$-versions we obtain the following dependency graph for the $K$-version (\cite{ford2}, Thm. 1.1):
\[
\text{(\cite{MR}, Thm. 6.2)} \Rightarrow  \text{(\cite{MR}, Thm. 4.4)} \Rightarrow   \text{(\cite{MR}, Thm. 3.1)} \Rightarrow \text{(\cite{MR}, Thm. 1.1)}. \tag{4.3}
\]
The graphs (4.2) and (4.3) can be combined to the graph:

\vspace{-3mm}
\begin{figure}[H]
\begin{minipage}{0.01\textwidth}
(4.4) 
\end{minipage}
\hfill
\begin{minipage}{0.95\textwidth}
\begin{center}
\includegraphics[scale=0.35]{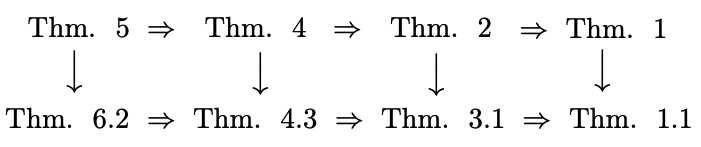}
\end{center}
\end{minipage}
\captionsetup{skip=-6pt}
\end{figure}
\vspace{-3mm}
(with Theorems 1, 2, 4, 5 corresponding to \cite{ford2} and Theorems 1.1, 3.1, 4.3, 6.2 corresponding to \cite{MR})

The horizontal arrows indicate the deduction of Thm. B from Thm. A, the vertical arrows indicate the transition from Thm. A to its $K$-version Thm. A'.\\
Part I of $``$Large gaps with improved order of magnitude and its $K$-version$"$ (Section 4) deals with the  graph (4.5). The end of the graph, Thm. 5 and its $K$-version, Thm. 6.2 is deduced from results of Maynard's paper \cite{maynard} $``$Dense clusters of primes in subsets$"$. The $K$-version, Thm. 6.2 is deduced from its $K$-version. These deductions make up Part II and are the contents of Section 5.\\
The graph (4.4) consists of segments, the last one being
\vspace{-3mm}
\begin{figure}[H]
\begin{minipage}{0.01\textwidth}
(4.5) 
\end{minipage}
\hfill
\begin{minipage}{0.95\textwidth}
\begin{center}
\includegraphics[scale=0.35]{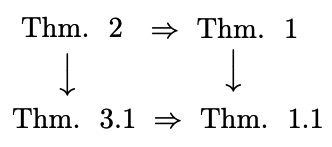}
\end{center}
\end{minipage}
\captionsetup{skip=-3pt}
\end{figure}
\vspace{-3mm}
(with Theorems 1, 2 corresponding to \cite{ford2} and Theorems 1.1, 3.1 corresponding to \cite{MR})

We shall proceed segment by segment starting with (4.5). in this way the transition from a theorem to its $K$-version should become more transparent. 

We start with the upper string in (4.5):
$$\text{Thm. 2}\ \Rightarrow\ \text{Thm. 1}.$$

Let $\vec{a}$ and $\vec{b}$ be as in (\cite{ford2}, Definitions (3.3)-(3.5)). We extend the tuple $(a_p)_{p\leq x}$ of congruence classes $a_p\bmod p$ for all primes $p\leq x$ by setting $a_p:=b_p$ for $p\in P$ and $a_p:=0$ for $p\not\in S\cup P$ and consider the sifted set
$$\mathcal{T}:=\{n\in [y]\setminus[x]\::\: n\not\equiv a_p (\bmod p)\ \text{for all}\ p\leq x\}\:.$$
As in previous versions, one shows that the second residual set consists of a negligible set of smooth numbers and the set $Q$ of primes. Thus we find that 
$$\#\mathcal{T}\ll\frac{x}{\log x}\:.$$
Next let $C$ be a sufficiently large constant such that $\#\mathcal{T}$ is less than the number of primes in $[x, Cx]$. By matching each of these surviving elements to a distinct prime in $[x, Cx]$ and choosing congruence classes appropriately, we thus find congruence classes $a_p \bmod p$ for $p\leq Cx$ which cover all of the integers in $(x,y)$. This finishes the deduction of Theorem 1 from Theorem 2.\\

\textit{$K$-version deduction of (\cite{MR}, Theorem 1.1) from (\cite{MR}, Theorem 3.1})\\
The first two sieving steps are the same as in the $``$upper string$"$ of (\cite{ford2}, \mbox{Thm. 2 $\Rightarrow$ Thm. 1}). Thus the second residual set is again $Q$ apart from a negligible set of smooth integers. The random  choice in the remaining sieving steps now has to be modified.\\
Let 
\begin{align*}
\mathcal{A}_{(K)}&:=\{\vec{a}=(a_s \bmod s)_{s\in S}\::\: \exists\ c_s\ \text{such that}\tag{4.17}\\
&\ \  a_s\equiv 1-(c_s+1)^K(\bmod s), c_s\not\equiv -1 (\bmod s)\}  
\end{align*}
\begin{align*}
\mathcal{B}_{(K)}&:=\{\vec{b}=(b_p \bmod p)_{p\in P}\::\: \exists\ d_p\ \text{such that}\tag{4.18}\\
&\ \  b_p\equiv 1-(d_p+1)^K(\bmod p), b_p\not\equiv -1 (\bmod p)\} . 
\end{align*}
One then has:
\begin{theorem}(\cite{MR}, Theorem 3.1)\label{thm431}\\
Let $x$ be sufficiently large and suppose that $y$ obeys (4.8). Then there are vectors $\vec{a}\in\mathcal{A}_{(K)}$ and $\vec{b}\in\mathcal{B}_{(K)}$, such that 
\[
\#(Q\cap S(\vec{a})\cap S(\vec{b}))\ll \frac{x}{\log x}\:. \tag{4.19}
\]
\end{theorem}
We now sketch the deduction of (\cite{MR}, Theorem 1.1) from (\cite{MR}, Theorem 3.1).\\
Let $\vec{a}$ and $\vec{b}$ be as in (\cite{MR}, Theorem 3.1). We extend the tuple $\vec{a}$ to a tuple $(a_p)_{p\leq x}$ of congruence classes $a_p(\bmod p)$ for all primes $p\leq x$ by setting $a_p:= b_p$ for $p\in P$ and $a_p:=0$ for $p\not\in S\cup P$. Again the sifted set 
$$\mathcal{T}:=\{n\in [y]\setminus[x]\::\: n\not\equiv a_p (\bmod p)\ \text{for all}\ p\leq x\}\:,$$
differs from the set $Q\cap S(\vec{a})\cap S(\vec{b})$ only by a negligible set of $z$-smooth integers. We find
\[
\#\mathcal{T}\ll\frac{x}{\log x}\:.\tag{\cite{MR}, Lemma 3.2}
\]
As in the $``$upper string deduction$"$ (\cite{ford2}, Thm. 2) $\Rightarrow$ (\cite{ford2}, Thm. 1) we now further reduce the sifted set $\mathcal{T}$ by using the prime numbers from the interval $[x, C_0x]$, $C_0>1$ being a sufficiently large constant.\\
One follows - with some modification in the notation - the papers \cite{konyagin, MR}. One distinguishes the cases $K$ odd and $K$ even. We recall the following definition:
\begin{definition}(\cite{MR}, Definition 3.3)\label{defn43}
Let
\begin{equation*}
    \tilde{P} = 
    \begin{cases}
      p\::\: & x<p\leq C_0x,\ p\equiv 2(\bmod\: 3),\ \text{if $K$ is odd}\\
      p\::\: & x<p\leq C_0x,\ p\equiv 3(\bmod\: 3K),\ \text{if $K$ is even.}
    \end{cases}
\end{equation*}
For $K$ even and $\delta>0$, we set
$$U=\{u\in[0,y]\::\:\left(\frac{-u}{p}\right)=1\ \text{for at most}\ \frac{\delta x}{\log x}\ \text{primes}\ p\in\tilde{P}\}.$$
\end{definition}
By \cite{konyagin} we have:\\
\begin{lemma}\label{lem4646} $\#U\ll_{\epsilon} x^{1/2+\epsilon}$.
\end{lemma}
\begin{lemma}\label{lem_eupu} There are pairs $(u, p_u)$ with $u\in\mathcal{T}$, $p_u\in\tilde{P}$, such that all $u\in\mathcal{T}$ satisfy a congruence
$$u\equiv 1-(e_u+1)^K(\bmod p_u)\ \text{where}\ e_u\not\equiv-1(\bmod\: p_u)$$
with the possible exceptions of $u$ from an exceptional set $V$ with 
$$\#V\ll x^{1/2+2\epsilon}.$$
\end{lemma}
\textit{Proof.} If $K$ is odd, the congruence 
$$u\equiv 1-(e_u+1)^K(\bmod \: p)\ (\text{with the variable $e_u$})$$
is solvable, whenever $p\equiv 2(\bmod\: 3)$.\\
If $K$ is even, the congruence is solvable whenever $p\equiv 3(\bmod 2K)$ and $\left(\frac{-u}{p}\right)=1$. The claim now follows from (4.20) and Lemma \ref{lem4646}.\qed\\
We now conclude the deduction of Theorem 1.1 by the application of the matrix method. The following definition is borrowed from \cite{maier}.
\begin{definition}\label{defn444}
Let us call an integer $q>1$ a $``$good$"$ modulus, if $L(s,\chi)\neq 0$ for all characters $\chi\bmod q$ and all $s=\sigma+it$ with 
$$\sigma>1-\frac{c_2}{\log(q(|t|+1))}\:.$$
This definition depends on the size of $c_2>0$.
\end{definition}
\begin{lemma}\label{lem444}
There is a constant $c_2>0$, such that, in terms of $c_2$, there exist arbitrarily large values of $x$, for which the modulus
$$P(x)=\prod_{p<x} p$$
is good.
\end{lemma}
\text{Proof.} This is Lemma 1 of \cite{maier}.\qed
\begin{lemma}\label{lem45}
Let $q$ be a good modulus. Then 
$$\pi(x;q,a)\gg \frac{x}{\phi(q)\log x}\:,$$
where $\phi(\cdot)$ denotes Euler's totient function, uniformly for $(a,q)=1$ and $x\geq q^D$. Here the constant $D$ depends only on on the value of $c_2$ in Lemma \ref{lem444}.
\end{lemma}
\textit{Proof.} This result, which is due to Gallagher \cite{gallagher}, is Lemma 2 from \cite{maier}. \qed

We now define the matrix $\mathcal{M}$.
\begin{definition}
Choose $x$, such that $P(C_0x)$ is a good modulus. Let $\vec{a}\in\mathcal{A}_{(K)}$ and $\vec{b}\in\mathcal{B}_{(K)}$ be given. From the definition of $\mathcal{A}_{(K)}$ and $\mathcal{B}_{(K)}$, there are 
$$(c_s \bmod\: s)_{s\in S}\ \ \text{and}\ \ (d_p \bmod\: p)_{p\in P}\:,\ c_s\not\equiv -1 (\bmod \: s)\:,\ d_p\not\equiv -1 (\bmod\: p)\:,$$
such that
$$\vec{a}=(1-(c_s+1)^K \bmod s)_{s\in S}\ \ \text{and}\ \ \vec{b}=(1-(d_p+1)^K\bmod p)_{p\in P}\:.$$
We now determine $m_0$ by
$$1\leq m_0< P(C_0x)$$
and the congruences
\begin{align*}
m_0&\equiv c_s (\bmod\: s)\\
m_0&\equiv d_p (\bmod\: p)\\
m_0&\equiv 0 (\bmod\: q)\:,\ q\in(1, x],\: q\not\in S\cup P\tag{4.21}\\
m_0&\equiv e_u (\bmod\: p_u)\:, (e_u, p_u)\ \text{given by Lemma \ref{lem_eupu}}\\
m_0&\equiv g_p (\bmod\: p)\:,\ \text{for all other primes $p\leq C_0x$, $g_p$ arbitrary}.
\end{align*}
\end{definition}
By the Chinese Remainder Theorem $m_0$ is uniquely determined. We let
$$\mathcal{M}=(a_{r,u})_{\substack{1\leq r\leq P(x)^{D-1}\\ 1\leq u\leq y}}$$
with 
$$a_{r,u}=(m_0+1+rP(x))^K+u-1\:.$$
For $1\leq r\leq P(x)^{D-1}$, we denote by
$$R(r)=(a_{r,u})_{0\leq u\leq y}$$
the $r$-th row of $\mathcal{M}$ and for $0\leq u\leq y$, we denote by
$$C(u)=(a_{r, u})_{1\leq r\leq P(x)^{D-1}}$$
the $u$-th column of $\mathcal{M}$.
\begin{lemma}\label{lem46}
We have that $a_{r,u}$, $2\leq u\leq y$ is composite unless $u\in V$.
\end{lemma}
\begin{proof}
From the congruences
$$m_0 \equiv c_s (\bmod\: s)\:,\ \text{resp.}\ m_0\equiv d_p (\bmod\: p)\:,\ m_0\equiv 0 (\bmod\: q)\:,\ m_0\equiv e_u (\bmod\: p_u)$$
in (4.21), it follows that for 
$$u\equiv 1-(d_p+1)^K (\bmod\: p)\:,\ u\equiv 0 (\bmod\: q)\:, u\equiv 1-(e_u+1)^K (\bmod\: p_u)$$
we have
$$a_{r,u}\equiv 0 (\bmod\: s)\:, \text{resp.}\ a_{r,u}\equiv 0 (\bmod\: p)\:,\ a_{r,u}\equiv 0 (\bmod\: q)\:,\ a_{r,u}\equiv 0 (\bmod\: p_u)\:.$$
\end{proof}

\begin{definition}\label{defn46}
Let 
$$\mathcal{R}_0(\mathcal{M}):=\{r\::\: 1\leq r\leq P(x)^{D-1}\:,\ m_0+1+rP(x)\ \text{is prime}\}\:,$$
$$\mathcal{R}_1(\mathcal{M}):=\{r\::\: 1\leq r\leq P(x)^{D-1}\:,\ r\in\mathcal{R}_0(\mathcal{M})\:, R(r)\ \text{contains a prime number}\}\:.$$
\end{definition}
\noindent\textit{Remark.} We observe that each $a_{r-1}$ row $R(r)$ with $r\in\mathcal{R}_0(\mathcal{M})$ has as its first element 
$$a_{r-1}=(m_0+1+rP(x))^K\:,$$
the $K$-th power of the prime $m_0+1+rP(x)$.\\
If $r\in \mathcal{R}_0(\mathcal{M})\setminus\mathcal{R}_1(\mathcal{M})$, $a_{r,1}$ is the $K$-th power of a prime of the desired kind.
To deduce Theorem \ref{thm11} from Theorem \ref{thm431} it thus remains to show that \mbox{$\mathcal{R}_0(\mathcal{M})\setminus\mathcal{R}_1(\mathcal{M})$} is nonempty.
\begin{lemma}\label{lem47} We have
$$\#\mathcal{R}_0(\mathcal{M})\gg \frac{P(x)^D}{\phi(P(x))\log(P(x)^D)}\:.$$
\end{lemma}
\begin{proof}
This follows from Lemma \ref{lem45}.
\end{proof}

We obtain an upper estimate for $\mathcal{R}_1(\mathcal{M})$ by the observation that, if $R(r)$ contains a prime number, then 
$$m_0+1+rP(x)\ \ \text{and}\ \ (m_0+1+rP(x))^K+v-1$$
are primes for some $v\in V$.\\
The number
\begin{align*}
t(v)=&\#\{r\::\: 1\leq r\leq P(x)^{D-1},\ m_0+1+rP(x)\\
&\ \ \text{and}\ (m_0+1+rP(x))^K+v-1\ \text{are primes}\}
\end{align*}
is estimated by standard sieves as in Lemma 6.1 of \cite{konyagin}.\\
This concludes the deduction of Theorem \ref{thm11} from Theorem \ref{thm431}.
We now come to the next section in the graph (4.5).
\vspace{-3mm}
\begin{figure}[H]
\begin{center}
\includegraphics[scale=0.35]{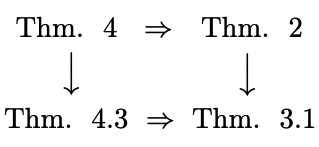}
\captionsetup{skip=-3pt}
\end{center}
\end{figure}
\vspace{-3mm}
We first state a hypergraph covering theorem (Theorem 3 of \cite{ford2}) of a purely combinatorial nature, generalizing a result of Pippenger and Spencer \cite{pip} using the R\"odl nibble method \cite{rodl}. We also state a corollary.\\
Both the deduction of Theorem 2 (Theorem \ref{thm42}) from Theorem 4 (Theorem \ref{thm44}) and its $K$-version, the deduction of Theorem \ref{thm431} from Theorem \ref{thm443} are based on Theorem 3 of \cite{ford2}.
%
%
%
%

\begin{theorem}\label{thm43}(Theorem 3 of \cite{ford2})(Probabilistic covering)\\
There exists a constant $C_0\geq 1$ such that the following holds. Let $D, r, A\geq 1$ and let \mbox{$0<\kappa\leq 1/2$,}  $m\geq 0$ be an integer. Let $\delta>0$ satisfy the smallness bound
\[
\delta\leq \left(\frac{\kappa^A}{C_0\exp(AD)}  \right)^{10^{m+2}} \tag{4.22}
\]
\end{theorem}
Let $I_1, \ldots, I_m$ be disjoint finite non-empty sets and let $V$ be a finite set. For each $1\leq j\leq m$ and $i\in I_j$, let $\bold{e}_i$ be a random finite subset of $V$. Assume the following:\\
$\bullet$ (Edges not too large) Almost surely for all $j=1,\ldots,m$ and $i\in I_j$, we have
$$\# \bold{e}_i\leq r_i$$
$\bullet$ (Each sieving step is sparse) For all $j=1,\ldots,m$, $i\in I_j$ and $v\in V$,
\[
\mathbb{P}(v\in \bold{e}_i)\leq \frac{\delta}{(\# I_j)^{1/2}} \tag{4.23}
\]
$\bullet$ (Very small codegrees) For every $j=1,\ldots, m$ and distinct $v_1, v_2\in V$,
\[
\sum_{i\in I_j} \mathbb{P}(v_1, v_2\in \bold{e}_i)\leq \delta \tag{4.24}
\]
$\bullet$ (Degree bound) If for every $v\in V$ and $j=1,\ldots, m$ we introduce the normalized degrees
\[
d_{I_j}(v):=\sum_{i\in I_j}\mathbb{P}(v\in \bold{e}_i) \tag{4.25}
\]
and then recursively define the quantities $P_j(v)$ for $j=0,\ldots, m$ and $v\in V$ by setting
\[
P_0(v):=1  \tag{4.26}
\]
and
\[
P_{j+1}(v):= P_j(v)\exp(-d_{I_{j+1}}(v)/P_j(v))  \tag{4.27}
\]
for $j=0, \ldots, m-1$ and $v\in V$, then we have 
\[
d_{I_j}(v)\leq DP_{j-1}(v),\ \ (1\leq j\leq m, v\in V) \tag{4.28}
\]
and
\[
P_j(v)\geq \kappa \ \ (0\leq j\leq m, v \in V)\:. \tag{4.29}
\]
Then we can find random variables $e_i'$ for each $i\in \bigcup_{j=1}^m I_j$ with the following properties:\\
(a) For each $i\in \bigcup_{j=1}^m I_j$, the essential support of $\bold{e}_i'$ is contained in the essential support of $\bold{e}_i$, union the empty set singleton $\{\emptyset\}$. In other words, almost surely $\bold{e}_i'$ is either empty, or is a set that $\bold{e}_i$ also attains with positive probability.\\
(b) For any $0\leq J\leq m$ and any finite subset $e$ of $V$ with $\# e\leq A-2rJ$, one has
\[ 
\mathbb{P}\left( e\subseteq V\setminus \bigcup_{j=1}^J\bigcup_{i\in I_j} \bold{e}_i'\right)=\left(1+O_{\leq}\left(\delta^{1/10^{J+1}}\right) \right)P_j(e) \tag{4.30}
\]
where
\[
P_j(e):=\prod_{v\in e} P_j(v)\:. \tag{4.31}
\]
The proof, which we will not give in this paper, is given in Section 5 of \cite{ford2}.\\
We have the following:
\begin{cor}\label{cor44} (Corollary 4 of \cite{ford2})\\
Let $x\rightarrow \infty$. Let $P', Q'$ be sets with $\# P'\leq x$ and $\# Q' > (\log_2 x)^3$. For each $p\in P'$, let $\bold{e}_p$ be a random subset of $Q'$ satisfying the size bound:
\[
\# \bold{e}_p \leq r=O\left(\frac{\log x\log_3 x}{\log_2^2 x}\right)\ \ (p\in P') \tag{4.32}
\]
Assume the following:\\
$\bullet$ (Sparsity) For all $p\in P'$ and $q\in Q'$
\[
\mathbb{P}(q\in \bold{e}_p)\leq x^{-1/2-1/10}\:.  \tag{4.33}
\]
$\bullet$ (Uniform covering) For all but at most $\frac{1}{(\log_2 x)^2}\# Q'$\  elements $q\in Q'$ we have:
\[
\sum_{p\in P'}\mathbb{P}(q\in \bold{e}_p)=C+O_{\leq}\left(\frac{1}{(\log_2 x)^2}\right)  \tag{4.34}
\]
for some quantity $C$, independent of $q$, satisfying
\[
\frac{5}{4}\log 5\leq C\leq 1\:.  \tag{4.35}
\]
$\bullet$ (Small codegrees) For any distinct $q_1, q_2\in Q'$
\[
\sum_{p\in P'} \mathbb{P}(q_1, q_2\in \bold{e}_p)\leq x^{-1/20}\:. \tag{4.36}
\] 
Then for any positive integer $m$ with 
\[
m\leq \frac{\log_3 x}{\log 5} \tag{4.37}
\]
we can find random sets $\bold{e}_p'\subseteq Q'$ for each $p\in P'$ such that
$$\#\{q\in Q'\::\: q\in \bold{e}_p'\ \text{for all}\ p\in P'\}\sim 5^{-m}\# Q'$$
with probability $1-o(1)$. More generally, for any $Q''\subset Q'$ with cardinality at least \mbox{$(\# Q')/\sqrt{\log_2 x}$} one has
$$\#\{q\in Q''\::\: q\not\in \bold{e}_p'\ \text{for all}\ p\in P'\}\sim 5^{-m}\# Q''$$
with probability $1-o(1)$. The decay rates in the $o(1)$ and $\sim$ notation are uniform in $P', Q', Q''$.\\
\end{cor}
\begin{proof}
For the proof we refer to \cite{ford2}.
\end{proof}
\begin{theorem}\label{thm44} (\cite{ford2}, Theorem 4)(Random construction)\\
Let $x$ be a sufficiently large real number and define $y$ by (4.8). Then there is a quantity $\mathcal{C}$ with
\[
\mathcal{C}\approx \frac{1}{c} \tag{4.38}
\]
with the implied constants independent of $c$, a tuple of positive integers $(h_1,\ldots, h_r)$ with $r\leq \sqrt{\log x}$, and some way to choose random vectors $\vec{\bold{a}}=(\bold{a}_s\ \bmod\: s)_{s\in S}$ and $\vec{\bold{n}}=(\bold{n}_p)_{p\in P}$ of congruence classes $\bold{a}_s \bmod s$ and integers $\bold{n}_p$ respectively, obeying the following:\\
$\bullet$ For every $\vec{a}$ in the essential range of $\vec{\bold{a}}$, one has
\[
\mathbb{P}(q\in \bold{e}_p(\vec{a})\ |\ \vec{\bold{a}}=\vec{a})\leq x^{1/2-1/10}\ \ (p\in P)\:,  \tag{4.39}
\]
where 
$$\bold{e}_p(\vec{a}):=\{\bold{n}_p+h_ip\::\: 1\leq i\leq r\}\cap Q\cap S(\vec{a})\:.$$
$\bullet$ With probability $1-o(1)$ we have that
\[
\#(Q\cap S(\vec{\bold{a}}))\sim 80\:c\:\frac{x}{\log x}\log_2 x.  \tag{4.40}
\]
$\bullet$  Call an element $\vec{a}$ in the essential range of $\vec{\bold{a}}$ good if, for all but at most $\frac{x}{\log x\log_2 x}$ elements $q\in Q\cap S(\vec{a})$, one has
\[
\sum_{p\in P}\mathbb{P}(q\in \bold{e}_p(\vec{a})\ |\ \vec{\bold{a}}=\vec{a})=\mathcal{C}+O_{\leq}\left(\frac{1}{(\log_2 x)^2} \right)  \tag{4.41}
\]
Then $\vec{\bold{a}}$ is good with probability $1-o(1)$.
\end{theorem}


We now show that Theorrem \ref{thm44} implies Theorem \ref{thm43}. By (4.38) we may choose $0<c<1/2$ small enough so that (4.35) holds. Take
$$m:=\left[ \frac{\log_3 x}{\log 5} \right]\:.$$
Now let $\vec{\bold{a}}$ and $\vec{\bold{n}}$ be the random vectors guaranteed by Theorem \ref{thm44}. Suppose that we are in the probability $1-o(1)$ event that $\vec{\bold{a}}$ takes a value $\vec{a}$ which is good and such that (4.40) holds. Fix some $\vec{a}$ within this event.
We may apply Corollary \ref{cor44} with $P' = P$ and $Q' = Q \cap S(\vec{a})$ for the random variables $\bold{n}_p$ conditioned to
$\vec{\bold{a}} = \vec{a}$. A few hypotheses of the Corollary must be verified. First, (4.34) follows from (4.10). The small
codegree condition (4.36) is also quickly checked. Indeed, for distinct $q_1, q_2\in  Q',$ if $q_1, q_2 \in e_p(\vec{a})$ then
$p|q_1-q_2$. But $q_1-q_2$ is a nonzero integer of size at most $x \log x$, and is thus divisible by at most one prime
$p_0 \in P'$. Hence
$$\sum_{p\in P'}\mathbb{P}(q_1, q_2\in \bold{e}_p(\vec{a}))=\mathbb{P}(q_1, q_2\in \bold{e}_{p_0}(\vec{a}))\leq x^{-1/2-1/10}\:,$$
the sum on the left side being zero if $p_0$ doesn't exist.\\ 
By Corollary \ref{cor44}, there exist random variables $\bold{e}'_p(\vec{a})$,
whose essential range is contained in the essential range of $\bold{e}_p(\vec{a})$ together with $\emptyset$, and satisfying
$$\{q\in Q\cap S(\vec{a})\::\: q\not\in \bold{e}_p'(\vec{a})\ \text{for all}\ p\in P\}\sim 5^{-m}\#(Q\cap S(\vec{a}))\ll \frac{x}{\log x}$$
with probability $1-o(1)$, where we have used (4.40). Since
$$\bold{e}_p'(\vec{a})=\{\bold{n}_p'+h_ip\::\: 1\leq i \leq r\}\cap Q\cap S(\vec{a})$$
for some random integer $\bold{n}_p'$, it follows that
$$\{q\in Q\cap S(\vec{a})\::\: q\not\equiv \bold{n}_p'(\bmod p)\ \text{for all}\ p\in P\}\ll \frac{x}{\log x}$$
with probability $1-o(1)$. Taking a specific $\vec{\bold{n}'} = \vec{n'}$ for which this relation holds and setting $b_p=n_p'$ for all
$p$ concludes the proof of the claim (4.17) and establishes Theorem \ref{thm42} (Theorem 2 of \cite{ford2}).

We now come to the $K$-version of the deduction Thm. 4 $\Rightarrow$ Thm. 2, the $``$lower string$"$ Thm. 4.3 $\Rightarrow$ Thm. 3.1 of the section
\vspace{-3mm}
\begin{figure}[H]
\begin{center}
\includegraphics[scale=0.35]{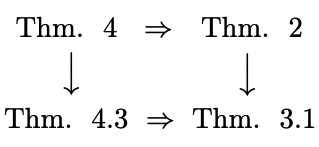}
\captionsetup{skip=-3pt}
\end{center}
\end{figure}
\vspace{-3mm}

\begin{theorem}\label{thm443} (\cite{MR}, Theorem 4.18 - Random construction)\\
Let $x$ be a sufficiently large real number and define $y$ by (4.8). Then there is a quantity $C$ with 
\[
C\asymp\frac{1}{c}
\]
with the implied constants independent of $c$, a tuple of positive integers
$(h_1,\ldots,h_r)$ with $r\leq \sqrt{\log x}$, and some way to choose random vectors 
$\vec{\bold{a}}=\: (\bold{a}_s\:\bmod\: s)_{s\in\mathcal{S}}$ and ${\vec{\bold{n}}}=(\bold{n}_p)_{p\in\mathcal{P}}$ of congruence classes $\bold{a}_s\:\bmod\: s$
and integers $\bold{n}_p$ respectively, obeying the following:
\begin{itemize}
\item For every $\vec{a}$ in the essential range of $\vec{\bold{a}}$, one has
$$\mathbb{P}(q\in\bold{e}_p(\vec{a})\:|\:\vec{\bold{a}}=\vec{a})\leq x^{-1/2-1/10}\ \ (p\in\mathcal{P})\:,$$
where $\bold{e}_p(\vec{a}):=\{ \bold{n}_p+h_ip\::\: 1\leq i\leq r \}\cap {Q}\cap S(\vec{a})$.
\item With probability $1-o(1)$ we have that 
\[
\#({Q}\cap S(\vec{\bold{a}}))\sim 80c\:\frac{x}{\log x}\: \log_2 x\:.
\]
\item Call an element $\vec{a}$ in the essential range of $\vec{\bold{a}}$ good if,
for all but at most $\frac{x}{\log x\log_2 x}$ elements $q\in{Q}\cap S(\vec{a})$, one has
\[
\sum_{p\in\mathcal{P}}\mathbb{P}(q\in\bold{e}_p(\vec{a})\:|\:\vec{\bold{a}}=\vec{a})=C+O_{\leq}\left(\frac{1}{(\log_2 x)^2} \right)\:.
\]
Then $\vec{\bold{a}}$ is good with probability $1-o(1)$.
\end{itemize}
\end{theorem}
\textit{Remark.} The wording of Theorem \ref{thm443} is the same as the wording of  (\cite{ford2}, Theorem 4). However, the contents of these two theorems are different, since the term \textit{essential range} has different meaning.\\
In Theorem \ref{thm44} $\vec{\bold{a}}$ resp. $\vec{\bold{b}}$ assume values of the form $\vec{a}\in \bigtimes_{s\in S} (a_s\bmod s)_{s\in S}$ resp. $\bigtimes_{p\in P} (b_p\bmod p)_{p\in P}$, whereas in Theorem \ref{thm443} they are of the form
$$\vec{a}=(a_s\bmod s,\ s\in S, \exists c_s\ \text{such that}\ a_s\equiv 1-(c_s+1)^K (\bmod s), c_s\not\equiv -1 (\bmod s))$$
$$\vec{b}=(b_p\bmod p,\ p\in P, \exists d_p\ \text{such that}\ b_p\equiv 1-(d_p+1)^K (\bmod p), d_p\not\equiv -1(\bmod p)).$$
Also the wording of the deduction of Theorem \ref{thm431} from Theorem \ref{thm443} is the same as the deduction of Theorem \ref{thm42} (Theorem 2 of \cite{ford2}) from Theorem \ref{thm44} (Theorem 4 of \cite{ford2}).\\
We come to the section:
\[
\textcolor{white}{}\tag{4.42}
\]
\vspace{-12mm}
\begin{figure}[H]
\begin{center}
\includegraphics[scale=0.35]{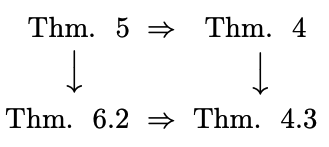}
\captionsetup{skip=-3pt}
\end{center}
\end{figure}
\vspace{-3mm}
of the graph (4.5).

The proof of this theorem relies on the estimates for multidimensional prime-detecting sieves established by the fourth author in \cite{ford2}. \\
We show now that Theorem \ref{thm45} implies Theorem \ref{thm44}.\\
Let $x, y, z, S, P, Q$ be as in Theorem \ref{thm44}. We set $r$ to be the maximum value permitted by Theorem \ref{thm45}, namely
\[ 
r=[\log^{1/5} x] \tag{4.50}
\]
and let $(h_1, \ldots, h_r)$ be the admissible $r$-tuple consisting of the first $r$ primes larger than $r$, thus $h_i=p_{\pi(r)+i}$ for $i=1, \ldots, r$. From the prime number theorem we have $h_i=O(r\log r)$ for $i=1,\ldots, r$ and so we have $h_i=[2r^2]$ for $i=1,\ldots, r$ if $x$ is large enough. We now invoke Theorem \ref{thm45} to obtain quantities $\tau, u$ and a weight $w\::\: P\times\mathbb{Z}\rightarrow\mathbb{R}^+$ with the stated properties. \\
For each $p\in P$, let $\tilde{\bold{n}}_p$ denote the random integer with probability density
$$\mathbb{P}(\tilde{\bold{n}}_p=n):=\frac{w(p, n)}{\sum_{n'\in\mathbb{Z}}w(p, n')}$$
for all $n\in\mathbb{Z}$ (we will not need to impose any independence condition on the $\tilde{\bold{n}}_p$. From (4.46), (4.47) we have
\[
\sum_{p\in P} \mathbb{P}(q=\tilde{n}_p+h_ip)=\left(1+O\left(\frac{1}{\log_2^{10} x}\right)\right)\:\frac{u}{r}\:\frac{x}{2y}\ \ (q\in Q, 1\leq i\leq r) \tag{4.51}
\]
Also from (4.46), (4.49), (4.44) one has
\[
\mathbb{P}(\tilde{\bold{n}}_p=n)\ll x^{-1/2-1/6+o(1)}  \tag{4.52}
\]
for all $p\in P$ and $n\in\mathbb{Z}$.\\
We choose the random vector $\vec{\bold{a}}:=(a_s \bmod\: s)_{s\in S}$ by selecting each $a_s  \bmod \: s$ uniformly at random from $\mathbb{Z}/s\mathbb{Z}$, independently in $s$ and independently of the $\tilde{n}_p$.\\
The resulting sifted set $S(\vec{\bold{a}})$ is a random periodic subset of $\mathbb{Z}$ with density 
$$\sigma:=\prod_{s\in S}\left(1-\frac{1}{s}\right)\:.$$
From the prime number theorem (with sufficiently strong error term), (4.9) and (4.10),
$$\sigma =\left(1+O\left(\frac{1}{\log_2^{10} x}\right)\right)\: \frac{\log(\log^{20} x)}{\log z}= \left(1+O\left(\frac{1}{\log_2^{10} x}\right)\right) \:\frac{80 \log_2 x}{\log x \log_2 x/\log_2 x}\:, $$
so in particular we see from (4.8) that
\[
\sigma y=\left(1+O\left(\frac{1}{\log_2^{10} x}\right)\right)\: 80 cx\log_2 x\:. \tag{4.53}
\]
We also see from (4.50) that 
$$\sigma^r=x^{o(1)}\:.$$
We have a useful correlation bound:
\begin{lemma}\label{lem449}
Let $t\leq \log x$ be a natural number and let $n_1, \ldots, n_t$ be distinct integers of magnitude $O(x^{O(1)})$. Then one has
$$\mathbb{P}(n_1, \ldots, n_t\in S(\vec{\bold{a}}))=  \left(1+O\left(\frac{1}{\log^{16} x}\right)\right)\sigma^t\:.$$
\end{lemma}
\begin{proof}
For each $s\in S$, the integers $n_1,\ldots, n_t$ occupy $t$ distinct residue classes modulo $s$, unless $s$ divides one of $n_i-n_j$ for $1\leq i< t$. Since $s\geq \log^{20} x$ and the \mbox{$n_i-n_j$} are of size $O(x^{O(1)})$, the latter possibility occurs at most \mbox{$O(t^2\log x)=O(\log^3 x)$} times. Thus the probability that $\vec{\bold{a}}\bmod s$ avoids all of the $n_1, \ldots, n_t$ is equal to $1-\frac{t}{s}$ except for $O(\log^3 x)$ values of $s$, where it is instead 
$$\left(1+O\left(\frac{1}{\log^{19} x}\right)\right)\left(1-\frac{t}{s}\right)\:.$$
Thus, 
\begin{align*}
\mathbb{P}(n_1, \ldots, n_t\in S(\vec{\bold{a}}))&=  \left(1+O\left(\frac{1}{\log^{19} x}\right)\right)^{O(\log^3 x)}\prod_{s\in S}\left(1-\left(\frac{t}{s}\right)\right)\\
&= \left(1+O\left(\frac{1}{\log^{16} x}\right)\right)\:\sigma^t\:\prod_{s\in S}\left(1+O\left(\frac{t^2}{s^2}\right)\right)\\
&= \left(1+O\left(\frac{1}{\log^{16} x}\right)\right)\:\sigma^t\:.
\end{align*}
\end{proof}
Among other things, this gives the claim (4.40):
\begin{cor}
With probability $1-o(1)$, we have
$$\mathbb{E}\#(Q\cap S(\vec{\bold{a}}))=\left(1+O\left(\frac{1}{\log^{16} x}\right)\right)\:\sigma\#Q$$
and
$$\mathbb{E}\#(Q\cap S(\vec{\bold{a}}))^2=\left(1+O\left(\frac{1}{\log^{16} x}\right)\right)\left(\sigma\#Q+\sigma^2(\#Q)(\#Q-1)\right)$$
and so by the prime number theorem we see that the random variable $\#Q\cap S(\vec{\bold{a}})$ has mean 
$$\left(1+o\left(\frac{1}{\log_2 x}\right)\right)\sigma\frac{y}{\log x}$$ 
and variance 
$$O\left(\frac{1}{\log^{16} x}\left( \sigma\: \frac{y}{\log x} \right)^2 \right).$$
The claim then follows from Chebyshev's inequality (with plenty of
room to spare). 
\end{cor}
For each $p\in P$, we consider the quantity 
\[
X_p(\vec{a}):=\mathbb{P}(\vec{\bold{n}}_p+h_ip\in S(\vec{a})\ \text{for all}\ i=1,\ldots, r), \tag{4.56}
\]
and let $P(\vec{a})$ denote  the set of all the primes $p \in P$ such that
\[
X_p(\vec{a}):=\left(1+O_{\leq}\left(\frac{1}{\log^3 x}\right)\right) \sigma^r. \tag{4.57}
\]
In light of Lemma \ref{lem449}, we expect most primes in $P$ to lie in $P(\vec{a})$, and this will be confirmed below
Lemma \ref{lem411}. We now define the random variables $\bold{n}_p$ as follows. Suppose we are in the event $\vec{\bold{a}} = \vec{a}$ for
some $\vec{a}$ in the range of $\vec{\bold{a}}$. If $p \in P\setminus P(\vec{a})$, 
we set $\bold{n}_p = 0$. Otherwise, if $p \in P(\vec{a})$, we define $n_p$ to be 
the random integer with conditional probability distribution
\[
\mathbb{P}(\vec{n}_p=n\:|\: \vec{\bold{a}}=\vec{a}):=\frac{Z_p(\vec{a};n)}{X_p(\vec{a})}\:,\ \ Z_p(\vec{a}; n)=1_{n+h_j p\in S(\vec{a})\ \text{for}\ j=1,\ldots r} \:\mathbb{P}(\tilde{\bold{n}}_p=n)\:,  \tag{4.58}
\]
with the $\bold{n}_p$ $(p \in P(\vec{a}))$ jointly independent, conditionally on the event $\vec{\bold{a}} = \vec{a}$. From (4.56) we see that these
random variables are well defined.
\begin{lemma}\label{lem410}
With probability $1-o(1)$, we have
\[
\sigma^{-r}\sum_{i=1}^r\sum_{p\in P(\vec{\bold{a}}} Z_p(\vec{\bold{a}}; q-h_ip)= \left(1+O\left(\frac{1}{\log_2^{3} x}\right)\right)\frac{u}{\sigma}\frac{x}{2y} \tag{4.59}
\]
for all but at most $\frac{x}{2\log x\log_2 x}$ of the primes $q\in Q\cap S(\vec{\bold{a}})$.
\end{lemma}
Let $\vec{a}$ be good and $q\in Q\cap S(\vec{{a}})$. Substituting definition (4.58) into the left hand side of (4.59), using
(4.57), and observing that $q = \bold{n}_p + h_ip$ is only possible if $p \in P(\vec{\bold{a}})$, we find that
\begin{align*}
\sigma^{-r}\sum_{i=1}^r\sum_{p\in P(\vec{{a}})} Z_p(\vec{{a}}; q-h_ip)&=\sigma^{-r} \sum_{i=1}^r\sum_{p\in P(\vec{{a}})}  X_p(\vec{a})\mathbb{P}(\bold{n}_p=q-h_ip\:|\: \vec{\bold{a}}=\vec{a})\\
&=\left(1+O\left(\frac{1}{\log^{3} x}\right)\right)\sum_{i=1}^r\sum_{p\in P(\vec{{a}})}\mathbb{P}(\bold{n}_p=q-h_ip\:|\: \vec{\bold{a}}=\vec{a})\\
&=\left(1+O\left(\frac{1}{\log^{3} x}\right)\right)\sum_{p\in P}\mathbb{P}(q\in\bold{e}_p(\vec{a})\:|\: \vec{\bold{a}}=\vec{a}).
\end{align*}
where 
$$\bold{e}_p(\vec{a})=\{\bold{n}_p+h_ip\::\: 1\leq i\leq r\}\cap Q\cap S(\vec{a})$$
is as defined in Theorem \ref{thm44} (Theorem 4 of \cite{ford2}). Relation (4.41) (that is, $\vec{\bold{a}}$ is good with probability $1-o(1)$) follows upon noting that by (4.45), (4.50) and (4.53),
$$C:=\frac{u}{\sigma}\frac{x}{2y}\sim\frac{1}{c}.$$
Before proving Lemma \ref{lem410}, we first confirm that $P\setminus P(\vec{\bold{a}})$ is small with high probability.
\begin{lemma}\label{lem411}
With probability $1-O(1/\log^3 x)$, $P(\vec{a})$ contains all but $O\left(\frac{1}{\log^3 x}\frac{x}{\log x}\right)$ of the primes $p\in P$. In particular,
$$\mathbb{E}\#P(\vec{\bold{a}})=\#P\left(1+O\left(\frac{1}{\log^3 x}\right)\right).$$
\end{lemma}
\begin{proof}
By linearity of expectation and Markov's inequality, it suffices to show that for each $p \in P$, we have $p \in P(\vec{\bold{a}})$ with probability $1-O(\frac{1}{\log^6 x})$. It suffices to show that
\[
\mathbb{E}X_p(\vec{\bold{a}})=\mathbb{P}(\tilde{\bold{n}}_p+h_ip\in S(\vec{\bold{a}})\ \text{for all}\ i=1,\ldots r)=\left(1+O\left(\frac{1}{\log^{12} x}\right)\right)\sigma^r  \tag{4.59}
\]
and
\[
\mathbb{E}X_p(\vec{\bold{a}})^2=\mathbb{P}(\tilde{\bold{n}}_p^{(1)}+h_ip, \tilde{\bold{n}}_p^{(2)}+h_ip\in S(\vec{\bold{a}})\ \text{for all}\ i=1,\ldots r)=\left(1+O\left(\frac{1}{\log^{12} x}\right)\right)\sigma^{2r}  \tag{4.60}
\]
where $\tilde{\bold{n}}_p^{(1)}$, $\tilde{\bold{n}}_p^{(2)}$ are independent copies of $\tilde{\bold{n}}_p$ that are also independent of $\vec{\bold{a}}$.
\end{proof}
The claim (4.59) follows from Lemma \ref{lem449} (performing the conditional expectation over $\tilde{\bold{n}}_p$ first). A
similar application of Lemma \ref{lem449} allows one to write the left-hand side of (4.60) as
$$\left(1+O\left(\frac{1}{\log^{16} x}\right)\right) \mathbb{E} \sigma^{\#\{\tilde{\bold{n}}_p^{(l)}+h_ip\::\: i=1,\ldots, r\:;\: l=1,2\}}.$$
From (4.51) we see that the quantity $\#\{\tilde{\bold{n}}_p^{(l)}+h_ip\::\: i=1,\ldots, r\:;\: l=1,2\}$ is equal to $2r$ with probability
$1-O(x^{-1/2-1/6+o(1)})$, and is less than $2r$ otherwise. The claim now follows from (4.53). \\
\textit{Proof of Lemma \ref{lem410}}. We first show that replacing $P(\vec{\bold{a}})$ with $P$ has negligible effect on the sum, with probability 
$1-o(1)$. Fix $i$ and substitute $n = q - h_ip$. By Markov's inequality, it suffices to show that
\[
\mathbb{E}\sum_{n}\sigma^{-r}\sum_{p\in P\setminus P(\vec{\bold{a}})} Z_p(\vec{\bold{a}}; n)=o\left(\frac{u}{\sigma}\frac{x}{2y}\frac{1}{r}\frac{1}{\log_2^3 x}\frac{x}{\log x\log_2 x}  \right). \tag{4.61}
\]
by Lemma \ref{lem449}, we have
\begin{align*}
\mathbb{E}\sum_{n}\sigma^{-r}\sum_{p\in P} Z_p(\vec{\bold{a}}; n)&=
\sigma^{-r}\sum_{p\in P} \sum_{n} \mathbb{P}(\tilde{\bold{n}}_p=n)\mathbb{P}(n+h_jp\in S(\vec{\bold{a}})\ \text{for}\ j=1,\ldots, r)\tag{4.62}\\
&=\left(1+O\left(\frac{1}{\log^{16} x}\right)\right)\# P.
\end{align*}
Next, by (4.56) and Lemma \ref{lem411} we have
\begin{align*}
\mathbb{E}\sum_{n}\sigma^{-r}\sum_{p\in P(\vec{\bold{a}})} Z_p(\vec{\bold{a}}; n)&=\sigma^{-r}\sum_{\vec{a}}\mathbb{P}(\vec{\bold{a}}=\vec{a})\sum_{p\in P(\vec{\bold{a}})} X_p(\vec{a})\\
&=\left(1+O\left(\frac{1}{\log^{3} x}\right)\right)\mathbb{E}\#P(\vec{\bold{a}})=\left(1+O\left(\frac{1}{\log^{3} x}\right)\right) \# P;
\end{align*}
subtracting, we conclude that the left-hand side of (4.61) is \mbox{$O(\#P/ \log^3 x) = O(x/ \log^4 x)$.} The claim then
follows from (4.8) and (4.42).
By (4.61), it suffices to show that with probability $1-o(1)$, for all but at most 
$\frac{x}{2 \log x \log_2 x}$ primes $q \in Q \cap S(\vec{\bold{a}})$, one has
\[
\sum_{i=1}^r \sum_{p\in P} Z_p(\vec{\bold{a}} ; q-h_i p)=\left(1+O_{\leq}\left(\frac{1}{\log_2^3 x}\right)\right) \sigma^{r-1}u\frac{x}{2y}.  \tag{4.63}
\]
Call a prime $q\in  Q$ \text{bad} if $q \in Q \cap S(\vec{\bold{a}})$ but (4.63) fails. Using Lemma \ref{lem449} and (4.51), we have
\begin{align*}
&\mathbb{E}\left[ \sum_{q\in Q\cap S(\vec{\bold{a}})}\sum_{i=1}^r \sum_{p\in P} Z_p(\vec{\bold{a}} ; q-h_i p)\right]\\
&\ \ \ =\sum_{q, i, p}\mathbb{P}(q+(h_j-h_i)p\in S(\vec{\bold{a}})\ \text{for all}\ j=1,\ldots, r)\mathbb{P}(\tilde{\bold{n}}_p=q-h_ip)\\
&\ \ \ =\left(1+O\left(\frac{1}{\log_2^{10} x}\right)\right) \frac{\sigma y}{\log x}\sigma^{r-1} u\frac{x}{2y}
\end{align*}
and
\begin{align*}
&\mathbb{E}\left[ \sum_{q\in Q\cap S(\vec{\bold{a}})}\left(\sum_{i=1}^r \sum_{p\in P} Z_p(\vec{\bold{a}} ; q-h_i p)\right)^2\right]\\
&=\sum_{\substack{p_1, p_2, q \\ i_1, i_2}} \mathbb{P}(q+(h_j-h_{i_l})p_l\in S(\vec{\bold{a}})\ \text{for all}\ j=1,\ldots, r; l=1,2)\\
&\ \ \times \mathbb{P}(\tilde{\bold{n}}^{(1)}_{p_1}=q-h_{i_1}p_1) \mathbb{P}(\tilde{\bold{n}}^{(2)}_{p_2}=q-h_{i_2}p_2)\\
&=\left(1+O\left(\frac{1}{\log_2^{10} x}\right)\right) \frac{\sigma y}{\log x}\left(\sigma^{r-1} u\frac{x}{2y} \right)^2,
\end{align*}
where $(\tilde{\bold{n}}^{(1)}_{p_1})_{p_1\in P}$ and $(\tilde{\bold{n}}^{(2)}_{p_2})_{p_2\in P}$ are independent copies of  $(\tilde{\bold{n}}_{p})_{p\in P}$ over $\vec{\bold{a}}$. In the last step we used the
fact that the terms with $p_1 = p_2$ contribute negligibly.\\
By Chebyshev's inequality it follows that the number of bad $q$ is 
$$\ll\frac{\sigma y}{\log x}\frac{1}{\log_2^3 x}\ll\frac{x}{\log x\log_2^2 x} $$
with probability $1-O(1/ \log_2 x)$. \qed

We now come to the $K$-version, the $``$lower string$"$ Thm. 6.2 $\Rightarrow$ Thm. 4.3 of the section (4.41).\\
Like in the $``$upper string$"$ in Theorem 5 of \cite{ford2}, a certain weight function $w$ is of importance.  The construction of $w$ will be modelled on  
the construction of the function $w$ in \cite{ford2}, Theorem 5.\\
The restrictions $\vec{a}\in \mathcal{A}_{(K)}$, $\vec{b}\in \mathcal{B}_{(K)}$ bring some additional complications. The function $w(p, n)$ will be different from zero only if $n$ belongs to a set $\mathfrak{G}(p)$ of $p$-good integers. The definition of $\mathfrak{G}(p)$ is based  on the set $\mathfrak{G}$ of good integers.
\begin{definition}\label{def412}
For $(u, K)=1$ we define 
$$S_u:=\{ s\::\: s\ \text{prime}, s\equiv u (\bmod\: K), (\log x)^{20}< s\leq z\}$$
$$ d(u)=(u-1, K), r^*(u)=\frac{1}{d(u)}\sum_{s\in S_u} s^{-1}.$$
For $n\in [x,y]$ let 
$$r(n,u):=\sum_{\substack{s\in S_u\::\: \exists c_s, n\equiv 1-(c_s+1)^K (\bmod s)\\ c_s\not\equiv -1 (\bmod s)}} s^{-1}\:.$$
We set 
\begin{align*}\mathfrak{G}:=&\{n\::\: n\in [x,y], |r(n,u)-r^*(u)|\leq (\log x)^{-1/40}\\ 
\ &\text{for all}\ u(\bmod K), (u, K)=1\}.
\end{align*}
For an admissible $r$-tuple to be specified later and for primes $p$ with $x/2<p<x$ we set
$$\mathfrak{G}(p):=\{ n\in \mathfrak{G}\::\: n+(h_i-h_l)p\in \mathfrak{G}, \forall i, l\leq r\}\:.$$
\end{definition}
\begin{theorem}\label{thm62}(Theorem 6.2 of \cite{MR}) (Existence of good sieve weights)\\
Let $x$ be a sufficiently large real number and let $y$ be any quantity obeying (4.8). Let $P, Q$ be defined by (4.10), (4.11). Let $r$ be a positive integer with 
\[
r_0\leq r\leq \log^{\eta} x \tag{4.64}
\]
for some sufficiently large absolute constant $r_0$ and some sufficiently small $\eta>0$.\\
Let $(h_1, \ldots, h_r)$ be an admissible $r$-tuple contained in $[2r^2]$. Then one can find a positive quantity 
$$r\geq x^{-o(1)}$$
and a positive quantity $u=u(r)$ depending only on $r$ with 
\[
u \asymp \log r \tag{4.65}
\]
and a non-negative function
$$w_{(K)}\::\: P\times \mathbb{Z}\rightarrow \mathbb{R}^{+}$$
supported on $P\times(\mathbb{P}\cap [-y, y])$ with the following properties:
\[
w_{(K)}(p, n)=0 \tag{4.66}
\]
unless $n\equiv 1-(d_p+1)^K (\bmod\: p)$ for some $d_p\in\mathbb{Z}$, \mbox{$d_p\not\equiv -1(\bmod\: p)$} and $n\in\mathfrak{G}(p)$.\\
Uniformly for every $p\in P$, one has
\[
 \sum_{n\in\mathbb{Z}} w_{(K)}(p, n)=\left(1+O\left(\frac{1}{\log_2^{10} x}\right)\right) \left(\tau\:\frac{y}{\log x}\right).\tag{4.67}
\]
Uniformly for every $q\in Q$ and $i=1, \ldots, r$ one has
\[
\sum_{p\in P} w_{(K)}(p, q-h_ip)= \left(1+O\left(\frac{1}{\log_2^{10} x}\right)\right)\tau \frac{u}{r}\frac{x}{2\log^r x}.\tag{4.68}
\]
Uniformly for every $h=O(y/x)$ that is not equal to any of the $h_i$, one has
\[
\sum_{q\in Q}\sum_{p\in P} w_{(K)}(p, q-hp)=O\left(\frac{1}{\log_2^{10} x}\tau\frac{xy}{\log^r x \log\log x}\right). \tag{4.69}
\]
Uniformly for all $p\in P$ and $z\in \mathbb{Z}$
\[
w_{(K)}(p,n)=O\left(x^{1/3+o(1)} \right). \tag{4.70}
\]
\end{theorem}
We now show how Theorem \ref{thm62} implies Theorem \ref{thm443}.\\
Let $x, c, y, z, S, P, Q$ be as in Theorem \ref{thm43}. We set  
$$\tau:=\lfloor (\log x)^{\eta_0}\rfloor\:,\ \sigma:=\prod_{s\in S}\left(1-\frac{1}{s}\right)\:. $$
 We now invoke Theorem \ref{thm62} to obtain quantities $\tau, u$ and weight \mbox{$w\::\: P\times\mathbb{Z}\rightarrow \mathbb{R}^+$} with the stated properties.\\
 For each $p\in P$, let $\tilde{\bold{n}}_p$ denote the random integer with probability density
 \[
 \mathbb{P}(\tilde{\bold{n}}_p=n)=\frac{w_{(K)}(p, n)}{\sum_{n'\in\mathbb{Z}} w_{(K)}(p, n')} \tag{4.71}
 \]
 for all $n\in\mathbb{Z}$. From (4.67), (4.68) we have
  \[
 \sum_{p\in P}\mathbb{P}(q=\tilde{\bold{n}}_p+h_ip)= \left(1+O\left(\frac{1}{\log_2^{10} x}\right)\right) \frac{u}{r}\frac{x}{2y}\ \ (q\in Q, 1\leq i\leq r). \tag{4.72}
 \]
 Also, from (4.67), (4.71), (4.65) one has
 $$\mathbb{P}(\tilde{\bold{n}}_p=n)\ll x^{-1/2-1/6+o(1)}$$
 for all $p\in P$ and $n\in\mathbb{Z}$.\\
 We choose the random vector $\vec{\bold{a}}:=(a_s \bmod\: s)_{s\in S}$ by selecting each $a_s\bmod\: s$ uniformly at random from $\mathcal{A}_s$ independently in $s$.
\begin{lemma}\label{lem413}
Let $t\leq (\log x)^{3\eta_0}$ be a natural number and let $n_1, \ldots, n_t$ be distinct integers from $\mathfrak{G}$. Then, one has 
$$\mathbb{P}(n_1,\ldots, n_t\in S(\vec{\bold{a}}))= \left(1+O\left(\frac{1}{\log_2^{10} x}\right)\right)\sigma^t.  $$
\end{lemma}
\begin{proof}
For $\vec{n}=(n_1,\ldots, n_t)$, let $\mathcal{K}(\vec{n})$ be the set of $s\in S$ for which $s\mid n_l-n_i$, for $i\neq l$. Then, since
$$n_i-n_l=O(x^{O(1)}),$$
we have 
$$|\mathcal{K}(\vec{n})|=O((\log x)^3).$$
Let $\vec{a}\in\mathcal{A}_s$, $1\leq u\leq K-1$, $(u, K)=1$. We write
$$\vec{a}_u=(a_{s_{1,u}}, \ldots, a_{s_{r_u,u}}),$$
where 
$$S_u\cap \mathcal{K}^c=\{ s_{1,u},\ldots, s_{r_u,u} \}\:.$$
We set
\begin{equation*}
    \epsilon(h, s) := 
    \begin{cases}
      & 1,\ \text{if $n_h\equiv 1-(c_{s,h}+1)^K (\bmod\: s)$ has a solution $c_{s,h}\not\equiv -1(\bmod\: s)$}\\
     & 0,\ \text{otherwise.}
    \end{cases}
\end{equation*}
\end{proof}
We have 
$$n_h\in S(\vec{a})\ \ \text{if and only if $n_h\in S(\vec{a},u)$, $\forall n, (u, K)=1.$}  $$
We now use certain well-known facts from the theory of $K$-th power residues.\\
There are
$$\frac{s_{i,u}-1}{d(u)}-1$$
possible choices for the $a_{s_i,u}$. From these, for each $h$, $1\leq h\leq t$ there are $\epsilon(h, s_{i,u})$ choices such that
$$a_{s_i,u}\equiv n_h (\bmod s_{i,u})\:.$$
Thus, the total number of choices for the $a_{s_i,u}$ for which not all $u_h\in S(\vec{a})$, \mbox{$(1\leq h\leq t)$} is 
$$\sum_{h=1}^t \epsilon(h, s_{i,u}).$$
Since the choices for the components $a_s$ are independent, we have
\begin{align*}
&\mathbb{P}(n_1,\ldots, n_t\in S(\vec{\bold{a}}))  \tag{4.73}\\
&=\prod_{u\::\:(u, K)=1}\prod_{s\in S_u}\left(\frac{s-1-d(u)}{d(u)}\right)^{-1}\left(\frac{s-1-d(u)}{d(u)}-\sum_{h=1}^t \epsilon(h,s)\right)\left(1+O\left(\frac{(\log x)^3}{z_0}\right)\right)\\
&= \prod_{u\::\:(u, K)=1}\prod_{s\in S_u} \left(1-d(u)s^{-1}\sum_{h=1}^t \epsilon(h, s)\right)(1+O(s^{-2}))(1+O((\log x)^{-17})).
\end{align*}
We have
 $$\prod_{s\in S_u} \left(1-d(u)s^{-1}\sum_{h=1}^t \epsilon(h, s)\right)=\exp\left(-\sum_{s\in S_u} d(u)s^{-1}\sum_{h=1}^t \epsilon(h, s)+O(s^{-2})\right).$$
 Since $n_h\in\mathfrak{G}$ for $1\leq h\leq t$, we have by the Definition for $\mathfrak{G}$:
 \[
\sum_{s\in S(u)} s^{-1}\epsilon(h, s)=\frac{1}{d(u)}\sum_{s\in S_u} s^{-1}+O((\log x)^{-1/40}).  \tag{4.74}
 \]
 From (4.73) and (4.74) we thus obtain
 $$\mathbb{P}(n_1,\ldots, n_t\in S(\vec{\bold{a}}))=\left(1+O\left(\frac{1}{(\log x)^{1/40}}\right)\right)\sigma^t.$$
 \begin{cor}(to Lemma \ref{lem413})\\
 With probability $1-o(1)$ we have:
 $$\#(Q\cap S(\vec{\bold{a}}))\sim \sigma\frac{y}{\log x}\sim 80c\frac{x}{\log x}\log_2 x.$$
 \end{cor}
 \begin{proof}
 From Lemma \ref{lem413} we have
 $$\mathbb{E}\#(Q\cap S(\vec{\bold{a}}))=\left(1+O\left(\frac{1}{(\log_2 x)^{5}}\right)\right)\sigma\# Q$$
 and
  $$\mathbb{E}\#(Q\cap S(\vec{\bold{a}}))^2=\left(1+O\left(\frac{1}{(\log_2 x)^{5}}\right)\right)(\sigma\# Q+\sigma^2(\# Q)(\#Q-1))$$
and so by the prime number theorem we see  that the random variable $\#(Q\cap S(\vec{\bold{a}}))$ has mean
$$\left(1+O\left(\frac{1}{(\log_2 x)^{5}}\right)\right)\sigma\frac{y}{\log x}$$
and variance 
$$O\left(\frac{1}{(\log_2 x)^{5}}\left(\sigma\frac{y}{\log x}\right)^2\right).$$
The claim then follows from Chebyshev's inequality.
 \end{proof}
 
 For each $p\in P$ we consider the quantity
 \[
X_p(\vec{a}):=\mathbb{P}(\tilde{n}_p+h_ip\in S(\vec{\bold{a}}), \ \text{for all}\ i=1,\ldots, r)   \tag{4.75}
 \]
 and let $P(\vec{a})$ denote the set of primes $p\in P$, such that
 $$X_p(\vec{a}):=\left(1+O\left(\frac{1}{(\log_2 x)^{10}}\right)\right)\sigma^r.$$
 We now define the random variables $\bold{n}_p$ as follows. Suppose we are in the event $\vec{\bold{a}}=\vec{a}$ for some $\vec{a}$ in the range of $\vec{\bold{a}}$. If $p\in P\setminus P(\vec{a})$, we set $n_p:=0$. Otherwise, if $p\in P(\vec{a})$, we define $\bold{n}_p$ to be the random integer with conditional probability distribution
 $$\mathbb{P}(\bold{n}_p=n\:|\: \vec{\bold{a}}=\vec{a}):=\frac{Z_p(\vec{a}; n)}{X_p(\vec{a})}\:,$$
 where
$$Z_p(\vec{a};n):=1_{n+h_jp\in S(\vec{a}),\ \text{for $j=1,\ldots, r$}}\:\mathbb{P}(\tilde{\bold{n}}_p=n)$$
with the $\tilde{\bold{n}}_p$ jointly conditionally independent on the event $\vec{\bold{a}}=\vec{a}$. 
\begin{lemma}\label{lem414}
With probability $1-o(1)$ we have
\[
\sigma^{-r} \sum_{i=1}^r\sum_{p\in P(\vec{a})}Z_p(\vec{a}; q-h_i p)= \left(1+O\left(\frac{1}{(\log_2 x)^{5}}\right)\right)\frac{u}{\sigma}\frac{x}{2y}\tag{4.76}
\]
for all but at most $x/(2\log x\log_2 x)$ of the primes $q\in Q\cap S(\vec{a})$.
\end{lemma}
Before proving Lemma \ref{lem414}, we first confirm that $p\in P\setminus P(\vec{a})$ is small with high probability.
\begin{lemma}\label{lem415}
With probability
$$1-O\left(\frac{1}{(\log_2 x)^{10}}\right)\:,$$
$P(\vec{a})$ contains all but
$$O\left(\frac{1}{\log^3 x}\frac{x}{\log x}\right)$$
of the primes $p\in P$. In particular 
$$\mathbb{E}\# P(\vec{\bold{a}})=\# P\left(1+O\left(\frac{1}{\log^3 x}\right)\right).$$
\end{lemma}
\begin{proof}
By linearity of expectation and Markov's inequality, it suffices that for each $p\in P$ we have $p\in P(\vec{\bold{a}})$ with probability
$$1-O\left(\frac{1}{(\log_2 x)^{20}}\right)\:.$$
By Chebyshev's inequality it suffices to show that 
\begin{align*}
\mathbb{E} X_p(\vec{\bold{a}})&=\mathbb{P}(\tilde{\bold{n}}_p+h_ip\in S(\vec{\bold{a}})\ \text{for all $i=1, \ldots, r$})\tag{4.77}\\
&=\left(1+O\left(\frac{1}{\log_2 x}\right)\right)\sigma^r 
\end{align*}
and
\begin{align*}
\mathbb{E} X_p(\vec{a})^2&=\mathbb{P}(\tilde{\bold{n}}_p^{(1)}+h_ip, \tilde{\bold{n}}_p^{(2)}+h_ip\in S(\vec{\bold{a}})\ \text{for all $i=1, \ldots, r$})\tag{4.78}\\
&=\left(1+O\left(\frac{1}{\log_2 x}\right)\right)\sigma^{2r}, 
\end{align*}
where $\tilde{\bold{n}}_p^{(1)}, \tilde{\bold{n}}_p^{(2)}$ are independent copies of $\tilde{\bold{n}}_p$ that are also independent of $\vec{\bold{a}}$.\\
To prove the claim (4.77) we first select the value $n$ for $\tilde{n}_p$ according to the  distribution (4.71):
$$ \mathbb{P}(\tilde{\bold{n}}_p=n)=\frac{w_{(K)}(p, n)}{\sum_{n'\in\mathbb{Z}} w_{(K)}(p, n')}.$$
Because of the property $w(p,n)=0$ if $n\not\in \mathfrak{G}(p)$ we have with probability 1:
$$n+h_ip\in\mathfrak{G}\ \ \text{for $1\leq i\leq r$.}$$
The relation (4.77) now follows from Lemma \ref{lem413} with $n_i=n+h_ip$, applying the formula for total probability
$$\mathbb{P}(\tilde{\bold{n}}_p+h_ip\in S(\vec{\bold{a}}))=\sum_{n}\mathbb{P}(\tilde{\bold{n}}_p+h_ip\in S(\vec{\bold{a}})\:|\: \tilde{\bold{n}}_p=n).$$
A similar application of Lemma \ref{lem413} allows one to write the left hand side of (4.78) as
$$\left(1+O\left(\frac{1}{(\log_2 x)^5}\right)\right)\mathbb{E}\:\sigma^{\# \{\tilde{\bold{n}}_p^{(l)}+h_ip\::\: i=1,2,\ldots, r, l=1,2\}}.$$
From (4.77) we see that the quantity
$$\# \{\tilde{\bold{n}}_p^{(l)}+h_ip\::\: i=1,2,\ldots, r, l=1,2\}$$
is equal to $2r$ with probability
$$1-O\left(x^{-1/2-1/6+o(1)} \right)$$
and is less than $2r$ otherwise.\\
The claim now follows from $\sigma^{-r}=x^{o(1)}$.
\end{proof} 

\textit{Proof of Lemma \ref{lem414}.} We first show that replacing $P(\vec{\bold{a}})$ with $P$ has negligible effect on the sum with probability $1-o(1)$. Fix $i$ and substitute $n:=q-h_ip$.\qed

By Lemma \ref{lem414} we have
\begin{align*}
\mathbb{E}\sum_n \sigma^{-r}\sum_{p\in P} Z_p(\vec{\bold{a}}; n)&=\sigma^{-r}\sum_{p\in P}\sum_n \mathbb{P}(\tilde{\bold{n}}_p=n)\mathbb{P}(n+h_ip\in S(\vec{\bold{a}})\ \text{for}\ j=1,\ldots r)\\
&=\left(1+O\left(\frac{1}{(\log_2 x)^{10}}\right)\right)\# P.
\end{align*}
Next by 
$$X_p(\vec{a})=\left(1+O\left(\frac{1}{\log^3 x}\right)\right)\sigma^r$$
and Lemma \ref{lem415} we have
\begin{align*}
\mathbb{E}\sum_{r} \sigma^{-r}\sum_{p\in P(\vec{\bold{a}})} Z_p(\vec{a}; n)&=\sigma^{-r}\sum_a \mathbb{P}(\vec{\bold{a}}=\vec{a})\sum_{p\in P(\vec{a})} X_p(\vec{a})\\
&=\left(1+O\left(\frac{1}{(\log_2 x)^{10}}\right)\right)\mathbb{E}\# P(\vec{\bold{a}})\\
&=\left(1+O\left(\frac{1}{\log^3 x}\right)\right)\# P.
\end{align*}
Subtracting, we conclude that the difference of the two expectations above is $O(\# P/\log_2 x)$. The claim then follows from (4.8) and (4.64).\\
By this it suffices to show that 
$$\sigma^{-r}\sum_{i=1}^r \sum_{p\in P} Z_p (\vec{\bold{a}}; q-h_i p)=1+O\left(\frac{1}{\log_2 x}\right)$$
 for all but at most $\frac{x}{2\log x \log_2 x}$ primes $q\in Q\cap S(\vec{a})$ one has
 \[
\sum_{i=1}^r \sum_{p\in P} Z_p(\vec{a}; q-h_i p)=\left(1+O_\leq\left(\frac{1}{(\log_2 x)^3}\right)\right)\sigma^{r-1} u \:\frac{x}{2y}.  \tag{4.79}
 \]
 We call a prime $q\in Q$ $``$bad$"$, if $q\in Q\cap S(\vec{a})$, but (4.79) fails. Using Lemma \ref{lem415} and (4.71) we have
 \begin{align*}
&\mathbb{E}\left(\sum_{q\in Q\cap S(\vec{\bold{a}})} \sum_{i=1}^r \sum_{p\in P} Z_p(\vec{\bold{a}}; q-h_i p)  \right)\\ \tag{4.80}
&= \sum_{q, i, p}\mathbb{P}(q\in (h_j-h_i)p\in S(\vec{\bold{a}})\ \text{for all}\ j=1,\ldots r)\mathbb{P}(\tilde{\bold{n}}_p=q-h_ip).
 \end{align*}
 By the definition of $\mathfrak{G}(p)$ we have
 $$\mathbb{P}(q+(h_j-h_i)p\in S(\vec{\bold{a}}))=0,$$
 unless $q\in \mathfrak{G}(p)$. By Definition \ref{def412} this means that \mbox{$q+(h_j-h_i)p\in\mathfrak{G}$}.\\
 We may thus apply Lemma \ref{lem415} with 
 $$n_j:=(q-h_ip)+h_j p$$
 and obtain for all $i$:
 $$\mathbb{P}(q+(h_i-h_j)p\in S(\vec{\bold{a}})\ \text{for all}\ j=1,\ldots, r)=\sigma^r \left(1+O\left(\frac{1}{(\log_2 x)^{10}}\right)\right).$$ 
With (4.79) we thus obtain
 \begin{align*}
&\mathbb{E}\left(\sum_{q\in Q\cap S(\vec{\bold{a}})} \sum_{i=1}^r \sum_{p\in P} Z_p(\vec{\bold{a}}; q-h_i p)  \right)\\ 
&=\left(1+O\left(\frac{1}{(\log_2 x)^{10}}\right)\right)\frac{\sigma y}{\log x} \sigma^{r-1} u \frac{x}{2y} \:,
 \end{align*}
 Next we obtain
  \begin{align*}
&\mathbb{E}\left(\sum_{q\in Q\cap S(\vec{\bold{a}})}\left( \sum_{i=1}^r \sum_{p\in P} Z_p(\vec{\bold{a}}; q-h_i p)^2 \right)  \right)\\ 
&= \sum_{\substack{p_1, p_2, q \\ i_1, i_2}} \mathbb{P}(q+(h_j-h_{i_l})p_l\in S(\vec{\bold{a}})\ \text{for}\ j=1,\ldots, r; l=1,2\\
&\ \ \times\mathbb{P}(\tilde{\bold{n}}_{p_1}^{(1)}=q-h_{i_1}p_1)\mathbb{P}(\tilde{\bold{n}}_{p_2}^{(2)}=q-h_{i_2}p_2)\\
&=\left(1+O\left(\frac{1}{(\log_2 x)^{10}}\right)\right)\frac{\sigma y}{\log x} \sigma^{r-1} u \frac{x}{2y} \:,
 \end{align*}
 where $(\tilde{\bold{n}}_{p_1}^{(1)})_{p_1\in P}$ and $(\tilde{\bold{n}}_{p_2}^{(2)})_{p_2\in P}$ are independent copies of $(\tilde{\bold{n}}_{p})_{p\in P}$ over $\vec{\bold{a}}$. In the last step we used the fact that the terms with $p_1=p_2$ contribute negligibly.\\
 By Chebyshev's inequality it follows that the number of bad $q$'s is
 \[
\ll \frac{\sigma y}{\log x}\frac{1}{\log_2^2 x}\ll \frac{x}{\log x\log_2^2 x}\:,\ \text{with probability}\ 1-O\left(\frac{1}{\log_2 x}\right).  \tag{4.81}
 \]
 We may now prove Theorem \ref{thm43}.\\
 The relation (4.40) is actually the Corollary to Lemma \ref{lem413}. In order to prove (4.3), we assume that $\vec{a}$ is good and $q\in Q\cap S(\vec{a})$.\\
 Substituting (4.75) into the left hand side of (4.76) using that $\sigma^{-r}=x^{o(1)}$ and observing that $q=n_i+h_ip$ is only possible if $p\in P(\vec{a})$, we find that 
\begin{align*}
\sigma^{-r}\sum_{i=1}^r \sum_{p\in P(\vec{a})} Z_p(\vec{a}; q-h_ip)&=\sigma^{-r}\sum_{i=1}^r \sum_{p\in P(\vec{a})} X_p(\vec{a})\mathbb{P}(\bold{n}_p=q-h_i p\:|\: \vec{\bold{a}}=\vec{a})\\
&=\left(1+O\left(\frac{1}{(\log_2 x)^{2}}\right)\right)\sum_{i=1}^r \sum_{p\in P(\vec{a})} \mathbb{P}(n_p=q-h_i p\:|\: \vec{\bold{a}}=\vec{a})\\
&=\left(1+O\left(\frac{1}{(\log_2 x)^{2}}\right)\right)\sum_{i=1}^r  \sum_{p\in P}\mathbb{P}(q\in e_p(\vec{a})\:|\: \vec{\bold{a}}=\vec{a})\:,
\end{align*}
where 
$$e_p(\vec{a})=\{ n_p+h_ip\::\: 1\leq i\leq r \}\cap Q\cap S(\vec{a})$$
 is as defined in Theorem \ref{thm43}.
The fact that $\vec{a}$ is good with probability $1-o(1)$ follows upon noticing that 
$$C:=\frac{u}{\sigma}\frac{x}{2y}\sim \frac{1}{\sigma}.$$
This concludes the proof of Theorem \ref{thm43}. \qed

\section{Large Gaps with improved order of magnitude and its $K$-version, Part II}

We first state definitions and results from $``$Dense clusters of primes in subsets$"$ by Maynard \cite{maynard}.\\
We make use of the notation given in Section 7: $``$Multidimensional Sieve Estimates$"$ of \cite{maynard}.
\begin{definition}\label{def51}
A linear form is a function $L\::\:\mathbb{Z}\rightarrow\mathbb{Z}$ of the form $L(n)=l_1n+l_2$ with integer coefficients $l_1, l_2$ and $l_1\neq 0$. Let $\mathcal{A}$ be a set of integers. Given a linear form $L(n)=l_1n+l_2$. We define the sets
\begin{align*}
\mathcal{A}(x)&:=\{n\in \mathcal{A}\::\: x\leq n\leq 2x\},\\
\mathcal{A}(x; q, a)&:=\{ n\in \mathcal{A}\::\: n\equiv a (\bmod\: q)\},\\
\mathcal{P}_{L, \mathcal{A}}(x)&:= L(\mathcal{A}(x))\cap \mathcal{P}\\
\mathcal{P}_{L, \mathcal{A}}(x; q, a)&:=L(\mathcal{A}(x; q, a))\cap \mathcal{P}
\end{align*}
for any $x>0$ and congruence class $a\bmod q$, and define the quantity
$$\phi(q):=\phi(|l_1|q)/\phi(|l_1|),$$
where $\phi$ is the Euler totient function.\\
A finite set $\mathcal{L}=\{L_1, \ldots, L_k\}$ of linear forms is said to be admissible if $\prod_{i=1}^k L_i(n)$ has no fixed prime divisor, that is for every prime $p$ there exists an integer $n_p$ such that $\prod_{i=1}^k L_i(n_p)$ is not divisible by $p$.
\end{definition}
\begin{definition}\label{def52}
Let $x$ be a large quantity, let $\mathcal{A}$ be a set of integers, $\mathcal{L}=\{L_1, \ldots, L_k\}$ a finite set of linear forms and $B$ a natural number. We allow $\mathcal{A}, \mathcal{L}, k, B$ to vary with $x$. Let $0<\theta<1$ be a quantity independent of $x_0$. Let $\mathcal{L}'$ be a subset of $\mathcal{L}$. We say that the tuple $(\mathcal{A}, \mathcal{L}, \mathcal{P}, B, x, \theta)$ obeys Hypothesis 1 at $\mathcal{L}'$ if we have the following three estimates:\\
(1) ($\mathcal{A}(x)$ is well-distributed in arithmetic progressions). We have
\[
\sum_{q\leq x^\theta} \max_a \left|\# \mathcal{A}(x; q, a)-\frac{\#\mathcal{A}(x)}{q} \right|  \ll \frac{\#\mathcal{A}(x)}{\log^{100 k^2} x}
\]
(2) ($\mathcal{P}_{L, \mathcal{A}} (x)$ is well-distributed in arithmetic progressions). For any $L\in \mathcal{L}'$, we have
$$\sum_{q\leq x^\theta, (q, P)=1} \max_{(L(a), q)=1}\left| \#\mathcal{P}_{L, \mathcal{A}}(x; q, a)-\frac{\#\mathcal{P}_{L, \mathcal{A}}(x)}{\phi_K(q)}\right| \ll \frac{\# \mathcal{P}_{L, \mathcal{A}}(x)}{(\log x)^{100k^2}} .$$
(3) ($\mathcal{A}(x)$ not too concentrated). For any $q<x^\theta$ and $a\in \mathbb{Z}$ we have
$$\#\mathcal{A}(x; q, a)\ll \frac{\#\mathcal{A}(x)}{q}.$$
\end{definition}
In \cite{maynard} this definition was only given in the case $\mathcal{L}'=\mathcal{L}$, but we will need the (mild) generalization to the case in which $\mathcal{L}'$ is a (possibly empty) subset of $\mathcal{L}$.\\
As is common in analytic number theory, we will have to address the possibility of a Siegel zero. As we want to keep all our estimates effective, we will not rely on Siegel's theorem or its consequences. Instead, we will rely on the Landau-Page theorem, which we now recall. Throughout, $\chi$ denotes a Dirichlet character.
\begin{lemma}\label{lem53} (Landau-Page Theorem)\\
Let $Q\geq 100$. Suppose that $L(s, \chi)=0$ for some primitive character $\chi$ of modulus at most $Q$, and some $s=\sigma+it$. Then either
$$1-\sigma \gg \frac{1}{\log(Q(1+|t|)}\:,$$
or else $t=0$ and $\chi$ is a quadratic character $\chi_Q$, which is unique. Furthermore, if $\chi_Q$ exists, then its conductor $q_Q$ is square-free apart from a factor of at most 4, and obeys the lower bound
$$q_Q\gg \frac{\log^2 Q}{\log_2 ^2 Q}.$$
\end{lemma}
\begin{proof}
See e.g. \cite{daven_2000}, Chapter 14. The final estimate follows from the bound 
$$1-\beta\gg q^{-1/2}\log^{-2} q$$
for a real zero $\beta$ of $L(s, \chi)$ with $\chi$ of modulus $q$, which can also be found in \cite{daven_2000}, Chapter 14. \\
We can then eliminate the exceptional character by deleting at most one prime factor of $q_Q$.
\end{proof}
\begin{cor}\label{corQQ}
Let $Q\geq 100$. Then there exists a quantity $B_Q$ which is either equal to 1 or is a prime of size
$$B_Q\gg \log_2 Q$$
with the property that
$$1-\sigma\gg \frac{1}{\log(Q(1+|t|))}$$
whenever $L(\sigma+it, \chi)=0$ and $\chi$ is a character of modulus at most $Q$ and coprime to $B_Q$.
\end{cor}
\begin{proof}
If the exceptional character $\chi_Q$ from Lemma \ref{lem53} does not exist, then take $B_Q:=1$, otherwise we take $B_Q$ to be the largest prime factor of $q_Q$. As $q_Q$ is square-free apart from a factor of at most 4, we have $\log q_Q\ll B_Q$ by  the prime number theorem and the claim follows.
\end{proof}
\begin{lemma}\label{lem54}
Let $x$ be a large quantity. Then there exists a natural number $B\leq x$, which is either 1 or a prime, such that the following holds.\\
Let $\mathcal{A}:=\mathbb{Z}$, let $\theta:=1/3$ and $\mathcal{L}:=\{L_1, \ldots, L_k\}$ be a finite set of linear forms \mbox{$L_i(n)=a_i+b_i$} (which may depend on $x$) with $k\leq \log^{1/5} x$, $1\leq |a_i|\leq \log x$, and $|b_i|\leq x\log^2 x$.\\
Let $x\leq y \leq x \log^2 x$, and let $\mathcal{L}'$ be a subset of $\mathcal{L}$ such that $L_i$ is non-negative on $[y, 2y]$ and $a_i$ is coprime to $B$ for all $L_i\in \mathcal{L}'$. Then $(\mathcal{A}, \mathcal{L}, P, B, y, \theta)$ obeys Hypothesis 1 at $\mathcal{L}'$ with absolute implied constants (i.e. the bounds in Hypothesis 1 are uniform over all such choices of $\mathcal{L}$ and $y$).
\end{lemma}
\begin{proof}
Parts (1) and (3) of Hypothesis 1 are easy to see; the only difficult verification is (2). We apply Corollary \ref{corQQ} with 
$$Q:=\exp(c_1\sqrt{\log x})$$
for some small absolute constant $c_1$ to obtain a quantity $B:=B_Q$ with the stated properties. By the Landau-Page theorem (see \cite{daven_2000}, Chapter 20), we have that if $c_1$ is sufficiently small then we have the effective bound
\[
\phi(q)^{-1}\sum^*_{\chi} |\psi(z, \chi)|\ll x\exp(-3c\sqrt{\log x})  \tag{5.1}
\]
for all $1<q<\exp(2c\sqrt{\log x})$ with $(q, B)=1$ and all $z\leq x\log^4 x$. Here the summation is over all primitive $\chi\mod q$ and 
$$\psi(z, \chi)=\sum_{n\leq x}\chi(n)\Lambda(n).$$
Following a standard proof of the Bombieri-Vinogradov Theorem (cf. \cite{daven_2000}, Chapter 28), we have (for a suitable constant $c>0$):
\begin{align*}
&\sum_{\substack{q< x^{1/2-\epsilon} \\ (q, B)=1}} \sup_{\substack{(a, q)=1 \\ z\leq x\log^4 x}} \left|\pi(z; q, a)-\frac{\pi(z)}{\phi(q)}  \right|  \tag{5.2}\\
&\ \ \ll x\exp(-c\sqrt{\log x})+\log x  \sum_{\substack{q< \exp(2c\sqrt{\log x}) \\ (q, B)=1}} \sum_\chi^* \sup_{z\leq x\log^4 x} \frac{|\psi(z, \chi)|}{\phi(q)}
\end{align*}
Combining these two statements and using the triangle inequality gives the bound required for (2).
\end{proof}

We now recall the construction of sieve weights from \cite{maynard}, Section 7.\\
Let 
$$W:=\prod_{\substack{p\leq 2k^2\\ p\nmid B}} p\:.$$
For each prime $p$ not dividing $B$, let 
$$r_{p, 1} (\mathcal{L})<\cdots <r_{p, \omega_{\mathcal{L}(p)}}(\mathcal{L})$$
be the elements $n$ of $[p]$ for which $$p\mid \prod_{i=1}^k L_i(n).$$
If $p$ is also coprime to $w$, then for each $1\leq a\leq  \omega_{\mathcal{L}(p)}$, let $j_{p, u}=j_{p,u}(\mathcal{L})$ denote the least element of $[k]$ such that 
$$p\mid L_{j_{p,u}} (r_{p,u}(\mathcal{L})).$$
Let $D_k(\mathcal{L})$ denote the set 
\begin{align*}
D_k(\mathcal{L})&:=\{ (d_1, \ldots, d_k)\in\mathbb{N}^k\::\: \mu^2(d_1\ldots d_k)=1 \::\:\\
& (d_1\ldots d_k, WB)=1;\ (d_j,p)=1\ \text{whenever}\ p\nmid BW\\ &\text{and}\ j\neq j_{p,1},\ldots, j_{p, \omega_{\mathcal{L}(p)}}\}.
\end{align*}
Define the singular series
$$\mathfrak{S}(\mathcal{L}):=\prod_{p\nmid B} \left(1-\frac{\omega_{\mathcal{L}}(p)}{p}\right)\left(1-\frac{1}{p}\right)^{-k}\:,$$
the function
$$\phi_{\omega_{\mathcal{L}}}:=\prod_{p\mid d} (p-\omega_{\mathcal{L}}(p)),$$
and let $R$ be a quantity of size
$$x^{\theta/10}\leq R\leq x^{\theta/3}.$$
Let $F\::\: \mathbb{R}^k \rightarrow \mathbb{R}$ be a smooth function supported on the simplex
$$R_k:=\{ (t_1, \ldots, t_k)\in \mathbb{R}_+^k\::\: t_1+\cdots+ t_k\leq 1   \}.$$
For any $(d_1, \ldots, d_k)\in D_k(\mathcal{L})$ define
$$Y_{(d_1, \ldots, d_k)}(\mathcal{L}):=\frac{1_{D_k(\mathcal{L})}(r_1, \ldots, r_k) W^k B^k}{\phi(WB)^k}\mathfrak{S}_{WB}(\mathcal{L})\: F\left(\frac{\log r_1}{\log R}, \ldots, \frac{\log r_k}{\log R}\right).$$
For any $(d_1, \ldots, d_k)\in D_k(\mathcal{L})$ define
$$\lambda_{(d_1, \ldots, d_k)}(\mathcal{L}):=\mu(d_1\ldots d_k)d_1\ldots d_k \sum_{d_i\mid r_i\ \text{for}\ i=1, \ldots, k} \frac{Y_{(r_1, \ldots, r_k)}(\mathcal{L})}{\phi_{\omega_{\mathcal{L}}}(r_1\cdots r_k)}\:,$$
and then define the function $w=w_{k, \mathcal{L}, B, R}\::\:\mathbb{Z}\rightarrow \mathbb{R}^+$ by 
\[
w(n):=\left( \sum_{d_1, \ldots, d_k\::\: d_i/L_i(n)\ \text{for all}\ i} \lambda_{(d_1, \ldots, d_k)}(\mathcal{L})  \right)^2\:.  \tag{5.3}
\]
We then have the following slightly modified form of Proposition 6.1 of \cite{maynard}.
\begin{theorem}\label{thm55}
Fix $\theta$, $\alpha>0$. Then there exists a constant $C$ depending only on $\theta, \alpha$ such that the following holds. Suppose that $(\mathcal{A}, \mathcal{L}, P, B, x, \theta)$ obeys Hypothesis $I$ at some subset $\mathcal{L}'$ of $\mathcal{L}$. Write $k:=\#\mathcal{L}$, and suppose that 
$x\geq C$, $B\leq x^\alpha$, and $C\leq k\leq \log^{1/5} x$. Moreover, assume that the coefficients $a_i, b_i$ of the linear forms $L_i(n)=a_i n+b_i$ in $\mathcal{L}$ obey  the size bound $|a_i|, |b_i|\leq x^\alpha$, and $C\leq k\leq \log^{1/5} x$. Moreover, assume that the coefficients $a_i$, $b_i$ of the linear forms $L_i(n)=a_i n+b_i$ in $\mathcal{L}$ obey the size bound $|a_i|, |b_i|\leq x^\alpha$ for all $i=1, \ldots, k$. Then there exists a smooth function $F\::\:\mathbb{R}^k\rightarrow \mathbb{R}$ depending only on $k$ and supported on the simplex $R_k$, and quantities $I_k$, $J_k$ depending only on $k$ with 
$$I_k\gg (2k\log k)^{-k}$$
and
\[
J_k \asymp \frac{\log k}{k}\: I_k  \tag{5.4}
\]
such that, for $w(n)$ given in terms of $F$ as above, the following assertions hold uniformly for $x^{\theta/10}\leq R\leq x^{\theta/3}$.\\
$\bullet$ We have
\[
\sum_{n\in\mathcal{A}(x)} w(n)=\left(1+O\left(\frac{1}{\log^{1/10} x}\right)\right)\frac{B^k}{\phi(B)^k}\mathfrak{S}(\mathcal{L}) \#\mathcal{A}(x)(\log R)^k I_k \:.\tag{5.5}
\]
$\bullet$ For any linear form $L(n)=a_L n+b_L$ in $\mathcal{L}'$ with $a_L$ coprime to $B$ and $L(n)> R$ on $[x, 2x]$, we have
\begin{align*}
&\sum_{n\in\mathcal{A}(x)}1_P(L(n))w(n)   \tag{5.6}\\
&= \left(1+O\left(\frac{1}{\log^{1/10} x}\right)\right) \frac{\Phi(|a_L|)}{|a_L|}\frac{B^{k-1}}{\phi(B)^{k-1}}\frac{B^{k-1}}{\phi(B)^{k-1}}\mathfrak{S}(\mathcal{L})\# P_{L, \mathcal{A}}(x)(\log R)^{k-1} J_h\\
&+ O\left(\frac{B^k}{\phi(B)^k} \mathfrak{S}(\mathfrak{L})\#\mathcal{A}(x)(\log R)^{k-1} I_h  \right)\:.
\end{align*}
$\bullet$ Let $L(n)=a_0 n+b_0$ be a linear form such that the discriminant 
$$\Delta_L:=|a_0|\prod_{j=1}^k |a_0b_j-a_jb_0|$$
is non-zero (in particular $L$ is not in $\mathcal{L}$). Then 
\[
\sum_{n\in \mathcal{A}(x)} 1_{P\cap [x^{\theta/10}, +\infty)} (L(n))w(n)\ll \frac{\Delta_L}{\phi(\Delta_L)}\frac{B^k}{\phi(B)^k}\mathfrak{S} (\mathcal{L})\#\mathcal{A}(x) (\log R)^{n-1} I_k.   \tag{5.7}
\]
$\bullet$ We have the crude upper bound
\[
w(n)\ll x^{2\theta/3+o(1)}  \tag{5.8}
\]
for all n$\in \mathbb{Z}$.
\end{theorem}
\begin{proof}
The first estimate (5.5) is given by \cite{maynard}, Proposition 9.1, (5.6) follows from \cite{maynard}, Proposition 9.2, in the case $(a_L, B)=1$, (5.7) is given by \cite{maynard}, Proposition 9.4, (taking $\xi:=\theta/10$ and $D:=1$), and the final statement (5.8) is given by part (iii) of \cite{maynard}, Lemma 8.5.
The bounds for $J_k$ and $I_k$ are given by \cite{maynard}, Lemma 8.6.\\
We can now prove Theorem \ref{thm55}. Let $x, y, r, h_1, \ldots, h_r$ be as in that theorem. We set
\begin{align*}
\mathcal{A}&:= \mathbb{Z},\\
\theta&:= 1/3,\\
k&:= r,\\
R&:=(x/4)^{\theta/3}
\end{align*}
and let $B=x^{o(1)}$ be the quantity from Lemma \ref{lem54}.\\
We define the function $w\::\: P\times \mathbb{Z}\rightarrow \mathbb{R}^+$ by setting 
$$w(p, n):=1_{[-y,y]} (n) w_{k,\mathcal{L}_p, B, R}(n)$$
for $p\in P$ and $n\in\mathbb{Z}$, where $\mathcal{L}_p$ is the (ordered) collection of linear forms $n\mapsto n+h_i p$ for $i=1, \ldots, r$ and $w_{k,\mathcal{L}_p, B, R}$ was defined in (5.3). Note that the admissibility of the $r$-tuple $(h_1, \ldots, h_r)$ implies the admissibility of the linear forms $n\mapsto n+h_i p$.\\
An important point is that many of the key components of $w_{k,\mathcal{L}_p, B, R}$ are essentially uniform in $p$. Indeed, for any primes, the polynomial 
$$\prod_{i=1}^k (n+h_ip)$$
is divisible by $s$ only at the residue classes - $h_ip\:\bmod\: s$. From this we see that 
$$\omega_{\mathcal{L}_p}(s):=\# \{ h_i(\bmod\: s)\}\ \text{whenever}\ s\neq p.$$
In particular, $\omega_{\mathcal{L}_p}(s)$ is independent of $p$ as long as $s$ is distinct from $p$, therefore
\begin{align*}
\mathfrak{S}(\mathcal{L}_p)&=\left(1+O\left(\frac{k}{x}\right)\right)\mathfrak{S}\:,\tag{5.9}\\
\mathfrak{S}_{BW}(\mathcal{L}_p)&=\left(1+O\left(\frac{k}{x}\right)\right)\mathfrak{S}_{BW}\:,
\end{align*}
for some $\mathfrak{S}, \mathfrak{S}_{BW}$ independent of $p$, with the error terms uniform in $p$. Moreover, if $s\nmid WP$ then $s> 2k^2$, so all the $h_i$ are distinct $\mod\: s$ (since the $h_i$ are less than $2k^2$). Therefore, if $s\nmid pWB$ we have $\omega_{\mathcal{L}_p}(s)=k$ and
$$\{ j_{s,1}(\mathcal{L}_p), \ldots, j_{s, \omega(s)}(\mathcal{L}_p)\}=\{ 1, \ldots, k \}.$$
Since all $p\in P$ are at least $x/2> R$, we have $s\neq p$ whenever $s\leq R$. From this we see that 
$$D_R(\mathcal{L}_p)\cap \left\{(d_1, \ldots, d_k)\::\: \prod_{i=1}^k d_i\leq R\right\}$$
is independent of $p$, and where the error term is independent of $d_1, \ldots, d_k$.\\
It is clear that $w$ is non-negative and supported on $P\times [-y, y]$, and from (5.8) we have (4.45). We set
\[
\tau:= 2\:\frac{B^k}{\phi(B)^k}\:\mathfrak{S}(\log R)^k(\log x)^k I_k \tag{5.11}
\]
and
$$u:=\frac{\phi(B)}{B}\:\frac{\log R\: k J_k}{\log x\: 2I_k}\:.$$
Since $B$ is either 1 or prime, we have
$$\frac{\phi(B)}{B} \asymp 1\:,$$
and from the definition of $R$ we also have 
\[
\frac{\log R}{\log x} \asymp 1. \tag{5.12}
\]
From (5.4) we thus obtain (4.45). From \cite{maynard}, Lemma 8.1(i) we have
$$\mathfrak{S}\geq x^{-o(1)}$$
and from \cite{maynard}, Lemma 8.6, we have
$$I_k= x^{o(1)}$$
and so we have the lower bound (4.44). (In fact we also have a matching upper bound) $\tau\leq x^{o(1)}$, but we will not need this).\\
It remains to verify the estimates (4.46) and (4.47). We begin with (4.46). Let $p$ be an element of $\mathcal{P}$. We shift the $n$ variable by $3y$ and rewrite
$$\sum_{n\in\mathbb{Z}} w(p, n) = \sum_{n\in \mathcal{A}(2y)} w_{k, \mathcal{L}_p-3y, B, R}(n)+O(x^{1-c+o(1)}),$$
where $ \mathcal{L}_p-3y$ denotes the set of linear forms $n\mapsto n+h_ip-3y$ for $i=1, \ldots, k$. (The $x^{1-c+o(1)}$ error arises from (4.48) and roundoff effect if $y$ is not an integer). This set of linear forms remains admissible, and 
$$\mathfrak{S}(\mathcal{L}_p-3y)=\mathfrak{S}(\mathcal{L}_p)=\left(1+O\left(\frac{k}{x}\right)\right)\mathfrak{S}.$$
The claim (4.46) now follows from (5.2) and the first conclusion (5.5) of 
Theorem \ref{thm55} (with $x$ replaced by $2y$, $\mathcal{L}'=\emptyset$, and $\mathcal{L}=\mathcal{L}_p-3y$), using Lemma \ref{lem54} to obtain Hypothesis 1.\\
Now we prove (4.47). Fix $q\in Q$ and $i\in\{1, \ldots, k\}$. We introduce the set $\tilde{\mathcal{L}}_{q, i}$ of linear forms $\tilde{\mathcal{L}}_{q, i, 1}, \ldots, \tilde{\mathcal{L}}_{q, i, k}$, where
$$\tilde{\mathcal{L}}_{q, i, i}:=n$$
and 
$$\tilde{\mathcal{L}}_{q, i, j}(n):= q+(h_j-h_i)n\ \ (1\leq j\leq k, \: j\neq i).$$
We claim that this set of linear forms is admissible. Indeed, for any prime $s\neq q$, the solutions of 
$$n\:\prod_{j\neq i} (q+(h_j-h_i)n) \equiv 0 (\bmod\: s)$$
are $n\equiv 0 $ and $n\equiv -y(h_j-h_i)^{-1} (\bmod\: s)\ \text{for}\ h_j\not\equiv h_i (\bmod \: s)\:,$ the number of which is equal to $\#\{h_j(\bmod\: s)\}$. Thus
\begin{align*}
\mathfrak{S}(\tilde{\mathcal{L}}_{q, i})&=\left(1+O\left(\frac{k}{x}\right)\right)\mathfrak{S}\:,\\
\mathfrak{S}_{BW}(\tilde{\mathcal{L}}_{q, i})&=\left(1+O\left(\frac{k}{x}\right)\right)\mathfrak{S}_{BW}\:,
\end{align*}
as before. Again, for $s\nmid WB$ we have that the $h_i$ are distinct $(\bmod\:s)$, and so if $s<R$ and $s\nmid WB$ we have $\omega_{\tilde{\mathcal{L}}_{q, i}(s)}=k$ and 
$$\{ j_{s, 1}(\tilde{\mathcal{L}}_{q, i}),\ldots, j_{s, \omega(s)}(\tilde{\mathcal{L}}_{q, i})\}= \{1, \ldots, k\}. $$
In particular 
$$ D_k(\tilde{\mathcal{L}}_{q, i})\cap \left\{(d_1, \ldots, d_k)\::\: \prod_{i=1}^k
d_i\leq R\right\}$$
is independent of $q, i$ and so
$$\lambda_{(d_1, \ldots, d_k)}(\tilde{\mathcal{L}}_{q, i})=\left(1+O\left(\frac{k}{x}\right)\right)\lambda_{(d_1, \ldots, d_k)}$$
where again the $O(k/x)$ error is independent of $d_1, \ldots, d_k$. From this, since $q-h_ip$ takes values in $[-y, y]$, we have that 
$$w_{k,\tilde{\mathcal{L}}, B, R}(p)=\left(1+O\left(\frac{k}{x}\right)\right)w_{k, \mathcal{L}_p, B, R}(q-h_ip)$$
whenever $p\in \mathcal{P}$ (note that the $d_i$ summation variable implicit on both sides of this equation is necessarily equal to 1). Thus, recalling that $P= \mathcal{P}\cap \left(\frac{x}{2}, x\right)$ we can write the left-hand side of (4.47) as 
$$\left(1+O\left(\frac{k}{x}\right)\right)\sum_{n\in \mathcal{A}(x/2)} 1_\mathcal{P} (\tilde{\mathcal{L}}_{q, i, i}(n) w_{k, \tilde{\mathcal{L}}_{q}, B, R}(n).$$
Applying the second conclusion on (5.6) of Theorem \ref{thm55} (with $x$ replaced by $x/2$, $\mathcal{L}'=\{\tilde{\mathcal{L}}_{q, i, i}\}$, and ${\mathcal{L}}=\tilde{\mathcal{L}}_{q, i}$) and using Lemma \ref{lem54} to obtain Hypothesis 1, this expression becomes
\begin{align*}
&\left(1+O\left(\frac{1}{\log_2^{10} x}\right)\right)\:\frac{B^{k-1}}{\phi(B)^{k-1}}\:\mathfrak{S}\#\mathcal{P}_{\tilde{\mathcal{L}}, q, i, i, \mathcal{A}}\left(\frac{x}{2}\right)(\log R)^{k+1} J_k\\
&+O\left(\frac{B^k}{\phi(B)^k}\:\mathfrak{S}\#\mathcal{A}\left(\frac{x}{2}\right)(\log R)^{k-1} I_k\right).
\end{align*}
Clearly $\#\mathcal{A}(x/2)=O(x)$, and from the prime number theorem one has
$$\#\mathcal{P}_{L_{q, i, i, \mathcal{A}}}\left(\frac{x}{2}\right)=\left(1+O\left(\frac{1}{\log_2^{10} x}\right)\right)\:\frac{x}{2\log x}$$
for any fixed $C>0$. Using (5.11), we can thus write the left-hand side of (5.6) as
$$\left(1+O\left(\frac{1}{\log_2^{10} x}\right)\right)\:\frac{u}{k}\:\tau\:\frac{x}{2\log^k x}\: \frac{x}{2\log^k x}+O\left(\frac{1}{\log R}\:\tau\:\frac{x}{\log^k x}\right)\:.$$
From (4.42), (4.44), the second error term may be absorbed into the first, and (4.46) follows.\\
Finally, we prove (4.47). Fix $h=O(y/x)$ not equal to any of the $h_i$, and fix $p\in \mathcal{P}$. By the prime number theorem, it suffices to show that 
$$\sum_{q\in Q} w(p, q-hp) \ll \frac{1}{\log_2^{10} x}\:\tau\: \frac{y}{\log^k x}\:.$$
By construction, the left-hand side is the same as 
$$\sum_{x-hp< n\leq y-hp} 1_{\mathcal{P}}(n+hp) w_{k,\mathcal{L}_p, B, R}(n)\:,$$
which we can shift as
$$\sum_{n\in\mathcal{A}(y-x)}1_{\mathcal{P}\cap[x^{\theta/10}, +\infty])}(n-y+2x) w_{k, \mathcal{L}_p-y+2x-hp, B, R}(n)+O(x^{1-c+o(1)}),$$
where again the $O(x^{1-c+o(1)})$ error is a generous upper bound for round off errors. This error is acceptable and may be discarded. Applying (5.7), we may then bound the main term by
\begin{align*}
&\ll\frac{\delta}{\phi(\Delta)}\:\frac{B^k}{\phi(B)^k}\:\mathfrak{G}(\mathcal{L}_p-y+2x-hp)y(\log R)^{k-1}I_k\\
&=\frac{\delta}{\phi(\Delta)}\:\frac{B^k}{\phi(B)^k}\:\mathfrak{G}(\mathcal{L}_p)y(\log R)^{k-1} I_k\:,
\end{align*}
where
$$\Delta:=\prod_{j=1}^k |hp-h_ip|\:.$$
Applying (5.10), (5.11), we may simplify the above upper bound as
$$\frac{\Delta}{\phi(\Delta)}\:\frac{y}{(\log R)(\log x)^k}\tau.$$
Now $h-h_i=O(y/x)=O(\log x)$ for each $i$, hence 
$\Delta\leq O(x(\log x)^k)$, and it follows from (5.9), and (4.43), observing $\frac{\log R}{\log x}\asymp 1.$
$$\frac{\Delta}{\phi(\Delta)} \ll \log_2 \Delta \ll \log_2 x\ll \frac{\log R}{\log_2^{10} x}\:.$$
This concludes the proof of Theorem \ref{thm55}, and hence Theorem \ref{thm1}.
\end{proof}

\textit{The $K$-version deduction of Theorem \ref{thm62} (of \cite{MR}).}

\noindent We now modify the weights $w_n$ to incorporate (for fixed primes $p$) the conditions
\[
n\equiv 1-(d_p+1)^K (\bmod\: p)\:,\ d_p\not\equiv -1 (\bmod \: p) \tag{5.13}
\]
and $n\in \mathcal{G}(p).$\\
We carry out the modification in two steps. In a first step we replace $w_n=w_n(\mathcal{L})$ by $w^*(p,n)=w^*(p, n, \mathcal{L})$. Here $p$ is a fixed prime with $x/2<p\leq x$.\\
Here we have to be more specific about the set $\mathcal{A}$. We set $\mathcal{A}:=\mathbb{Z}$.
\begin{definition}
Let $w_n$ be as in (5.3), $\mathcal{A}=\mathbb{Z}$, $p$ a fixed prime with $x/2<p\leq x$. Let also $D=(K-1, p)$. We set
\begin{equation*}
    w^*(p,n) := 
    \begin{cases}
      & Dw_n,\ \text{if there is $d_p\in\mathbb{Z}$}\ \text{with}\ n\equiv 1-(d_p+1)^K(\bmod\: p),\tag{*}\\
      &\ \ \ \ \ \ \ \ \ \ \ \ \ \ \  \ \ \ d_p\not\equiv -1(\bmod\: p)\\
     & 0,\ \text{otherwise.}
    \end{cases}
\end{equation*}
\end{definition}

We first express the solvability of (*) by the use of Dirichlet characters.

\begin{lemma}\label{lem56}
Let $p$ be a prime number. Let $D=(p-1, K)$, $\chi_0$ the principal character $\bmod\: D$. There are $D-1$ non-principal characters $\chi_1, \ldots, \chi_{D-1}\: \bmod\: D$, such that for all $n\in\mathbb{Z}$ we have
\begin{equation*}
    \frac{1}{D}\sum_{l=0}^{D-1}\chi_l(1-n)= 
    \begin{cases}
      & 1,\ \text{if $n\equiv 1-c^K (\bmod p)$ is solvable with $p\nmid c$}\\
     & 0,\ \text{otherwise.}
    \end{cases}
\end{equation*}
\end{lemma}
\begin{proof}
Let $\rho$ be a primitive root $\bmod\: p$,
$$1-n\equiv \rho^s (\bmod\: p),\ 0\leq s\leq p-2.$$
Setting
$$c\equiv \rho^y(\bmod\: p)$$
we see that the congruence  
\[
c^K \equiv 1-n (\bmod\: p) \tag{5.14}
\]
is solvable if and only if
\[
Ky\equiv s(\bmod\: p-1) \tag{5.15}
\]
has a solution $y$.\\
By the theory of linear congruences, this is equivalent to $D\mid s$. We have
\begin{equation*}
    \frac{1}{D}\sum_{l=0}^{D-1}e\left(\frac{ls}{D}\right)= 
    \begin{cases}
      & 1,\ \text{if $D\mid s$,}\\
     & 0,\ \text{otherwise.}
    \end{cases}
\end{equation*}
We now define the Dirichlet character $\chi_l$, ($0\leq l\leq D-1$),
$$\chi_l(-n)= e\left(\frac{ls}{D}\right)$$
and obtain the claim of Lemma \ref{lem56}.
\end{proof} 
\begin{theorem}\label{thm57}
Let $p, w^*(p, n), D$, as in the Definition of $w^*(p,n)$, $\mathcal{A}:=\mathbb{Z}$. Then we have
$$\sum_{n\in\mathcal{A}(x)}w^*(p,n)=\left(1+O\left(\frac{1}{(\log x)^{1/10}}\right)\right)\:\frac{B^k}{\phi(B)^k}\:\mathfrak{G}_B(\mathcal{L})\mathcal{A}(x)(\log R)^k I_k(F).$$
\end{theorem}
\begin{proof}
By Lemma \ref{lem56} we have
$$\sum_{n\in\mathcal{A}(x)} w^*(p,n)=\sum_{l=0}^{D-1}\sum_{n\in\mathcal{A}(x)} w_n\chi_l(1-n).$$
The sum belonging to the principal character
$$\chi_0=\sum_{n\in\mathcal{A}(x)} w_n\chi_0(1-n)$$
differs from the sum
$$\sum_{n\in\mathcal{A}(x)} w_n$$
only by $O(x^{1/2})$, since there are only $\frac{|\mathcal{A}(x)|}{p}$ terms with $n\equiv 1 (\bmod\: p)$, each of them has size at most $x^{1/3}$. We therefore have
\[
\sum_{n\in\mathcal{A}(x)} w_n \chi_0(1-n)=\sum_{n\in \mathcal{A}(x)} w_n+O(x^{1/2}).  \tag{5.16}
\]
Let now $1\leq l\leq D-1$. Here we closely follow the proof of Proposition 9.1 of \cite{maynard}. We split the sum into residue classes $n\equiv v_0 (\bmod W)$. We recall that 
$$W=\prod_{\substack{p\leq 2g^2\\ p\nmid B}} p <\exp ((\log x)^{2/5}).$$
If $$\left(\prod_{i=1}^g L_i(v_0), W\right)\neq 1,$$ then we have $w_n=0$ and so we restrict our attention to $v_0$ with $$\left(\prod_{i=1}^g L_i(v_0), W\right)= 1.$$
We substitute the definition of $w_n$, expand the square and swap the order of summation. This gives
$$\sum_{n\in\mathcal{A}(x)}\chi_l(1-n)=\sum_{v_0(\bmod W)}\sum_{d, e\in D_g} \lambda_d \lambda_e \sum_{\substack{n\in\mathcal{A}(x)\\ n\equiv v_0(\bmod W)\\ [d_i, e_i]\mid L_i(n), \forall i}} \chi_l(1-n).$$
The congruence conditions in the inner sum may be combined via the Chinese Remainder Theorem by a single congruence condition
$$1-n \equiv c(\bmod v),\ \text{where $v=W[d, e]$,}$$
where $[\cdot, \cdot]$ stands for the least common multiple.\\
There are $w\leq v$ Dirichlet characters $\psi_1, \ldots, \psi_w (\bmod W)$ such that 
$$1-n\equiv c\: (\bmod v)\ \ \text{if and only if}\ \ \frac{1}{w}\sum_{l=1}^w \overline{\psi(c)}\psi_l(1-n)=1.$$
We thus may write
$$\left|  \sum_{\substack{n\in\mathcal{A}(x)\\ n\equiv v_0(\bmod w)\\ [d_i, e_i]\mid L_i(n), \forall i}} \chi_l(1-n)  \right| \leq A\sum_{l=1}^z \left| \sum_{n\in I} \xi_l(1-n)  \right|,$$
with a suitable absolute constant $A$, an interval $I$ of length 
$$|I|\leq x(\log x)^2$$
and the $D(v)$ non-principal Dirichlet characters $\xi_{j, l}=\chi_j \psi_l$ of conductor $\geq p$ and modulus $\leq xv$.\\
By the P\'olya-Vinogradov bound we obtain:
\[
 \sum_{\substack{n\in\mathcal{A}(x)\\ -n\equiv v_0(\bmod w)\\ [d_i, e_i]\mid L_i(n), \forall i}} \chi(1-n)\ll x^{1/2} v.  \tag{5.17}
\]
The claim of Theorem \ref{thm57} now follows from (5.16) and (5.17).
\end{proof}

As a preparation for the proof pf Theorem \ref{thm59} which is a modification of Proposition 9.2 of \cite{maynard}, we state a Lemma on character sums over shifted primes.

\begin{lemma}\label{lem58}
Let $\chi$ be a Dirichlet character  $(\bmod q)$. Then for $N\leq q^{16/9}$ we have
$$\sum_{n\leq N}\Lambda(n)\chi(n+a)\leq (N^{7/8} q^{1/9}+N^{33/32}q^{-1/18}) q^{o(1)}.$$
\end{lemma}
\begin{proof}
This is Theorem 1 of \cite{fried}.
\end{proof}

\begin{theorem}\label{thm59}
Let $\mathcal{A}=\mathbb{Z}$,
$$L(n)=a_{m}n+b_m\in\mathcal{L}$$
satisfy $L(n)>R$ for $n\in[x, 2x]$ and
$$\sum_{\substack{q< x^\theta \\ (q, B)=1}}\max_{L(a, q)=1}\left|\#P_{L, \mathcal{A}}(x; q, a)-\frac{\#P_{L, \mathcal{A}}(x)}{\phi_L(q)}  \right| \ll 
\frac{\#P_{L, \mathcal{A}}(x)}{(\log x)^{100 g^2}}.$$
Then we have for sufficiently small $\theta$:
\begin{align*}
\sum_{n\in\mathcal{A}(x)} 1_P(L(n)) w^*(p, n)&=\left(1+O\left(\frac{1}{(\log x)^{1/10}}\right)\right)\frac{B^{g-1}}{\phi(B)^{g-1}}\mathfrak{S}(\mathcal{L})\\
&\times\# P_{L, \mathcal{A}}(x)(\log R)^{g+1} J_g(F)\prod_{\substack{p\mid a_m\\ p\nmid B}}\frac{p-1}{p}\\
&+O\left(\frac{B^g}{\phi(B)^g}\mathfrak{S}_B(\mathcal{L})\#\mathcal{A}(x)(\log R)^{g-1} I_g(F)\right).
\end{align*}
\end{theorem}
\begin{proof}
By Lemma \ref{lem56} we have
$$\sum_{n\in\mathcal{A}(x)} 1_P(L(n))w^*(p, n)=\frac{1}{D}\sum_{l=0}^{D-1}\sum_{n\in\mathcal{A}(x)} 1_P(L(n))w_n\chi_l(1-n).$$
The sum belonging to the principal character $\chi_0$ differs from the sum
$$\sum_{n\in\mathcal{A}(x)} 1_P(L(n))w_n$$
only by $O(\#\mathcal{A}(x))p^{-1}$ and thus in \cite{maynard}, Proposition 9.2, we have
\begin{align*}
\sum_{n\in\mathcal{A}(x)}1_P(L(n)) w^*(p, n)\chi_0(1-n)&=\left(1+O\left(\frac{1}{(\log x)^{1/10}}\right)\right)
\frac{B^{g-1}}{\phi(B)^{g-1}}\mathfrak{S}_B(\mathcal{L})\tag{5.18}\\
&\times\# P_{L, \mathcal{A}}(x)(\log R)^{g+1} I_g(F)\prod_{\substack{p\mid a_m\\ p\nmid B}}\frac{p-1}{p} \\
&+O\left(\frac{B^g}{\phi(B^g)}\mathfrak{S}_B(\mathcal{L})\#\mathcal{A}(x)(\log R)^{g-1} I_g(F)\right).
\end{align*}
For $1\leq l\leq D-1$ we follow closely the proof of Proposition 9.2 in \cite{maynard}. We again split the sum into residue classes $n\equiv v_0(\bmod W).$
If 
$$\left(\prod_{i=1}^g L_i(v_0), W\right)>1,$$ then we have $w_n=0$ and so we restrict our attention to $v_0$ with 
$$\left(\prod_{i=1}^g L_i(v_0), W\right)=1.$$ We substitute the definition of $w_n$, expand the square and swap the order of summation. Setting $\tilde{n}=n-1$, we obtain
\[
\sum_{\substack{n\in\mathcal{A}(x)}\\ n\equiv v_0 (\bmod W)} 1_p(L(n)) w_n\chi_l(1-n)= \sum_{d, e}\chi_d \lambda_l   \sum_{\substack{n\in\mathcal{A}(x)\\ n\equiv v_0(\bmod W)\\ [d_i, e_i]\mid L_i(n), \forall i}}1_p(L(n))  \chi_l(1-n)  \tag{5.19}
\]
If $\tilde{n}$ runs through the arithmetic progression 
$$\tilde{n}=Wh+v_0\ (h\in I_0),$$
then also $L(\tilde{n}+1)$ runs through an arithmetic progression
$$L(\tilde{n}+1)=a_nWh+a_m(v_0+1)+b.$$
Thus, we have
\begin{align*}
&\sum_{\substack{n\in\mathcal{A}(x)\\ n\equiv v_0 (\bmod W)}} 1_p(L(n)) \chi_l(1-n)\\
&\ \ \ =\sum_{\substack{\tilde{p}\equiv a_m(v_0+1)+b(\bmod a_mW)\\ \tilde{p}\ \text{prime},\ \tilde{p}\in I}} \chi_l(\tilde{p}+a_m(v_0+1)+b).
\end{align*}
Also the condition $\tilde{p}\equiv a_m(v_0+1)+b(\bmod a_mW)$ may be expressed with the help of Dirichlet characters
$$\omega_1,\ldots, \omega_{\phi(|a_mW|)}\: (\bmod\:|a_mW|)\:,$$
using orthogonality relations.\\
Theorem \ref{thm59} thus follows from (5.18) and Lemma \ref{lem58}.
\end{proof}

For the definition of the  weight $w_{(K)}(p,n)$ whose existence is claimed in Theorem \ref{thm62} we now have to be more specific about the set $\mathcal{L}$ of linear forms.

\begin{definition}\label{defn510}
Let the tuple $(h_1, \ldots, h_r)$ be given. For $p\in P$ and $n\in \mathbb{Z}$ let $\mathcal{L}_p$ be the (ordered) collection of linear forms $n\mapsto n+h_ip$ for $i=1, \ldots, r$ and set
\begin{equation*}
    w_{(K)}= 
    \begin{cases}
      & w^*(p, n, \mathcal{L}_p),\ \text{if $n\in \mathcal{G}(p)$,}\\
     & 0,\ \text{otherwise.}
    \end{cases}
\end{equation*}
\end{definition}
In the sequel we now show that in the sums 
$$\sum_{n\in\mathbb{Z}} w_{(K)}(p, n)\ \ \text{resp.}\ \ \sum_{p\in \mathcal{P}} w_{(K)}(p, q-h_ip)$$
appearing in (4.66) resp. (4.67) of Theorem \ref{thm62}, the function 
$w_{(K)}(p, \cdot)$ may be replaced by the function $w^*(p, \cdot, \mathcal{L}_p)$ with a negligible error.\\
Since these sums have been treated in Theorem \ref{thm57} resp. Theorem \ref{thm59}, this will essentially conclude the proof of Theorem \ref{thm62} and thus of Theorem \ref{thm11}. \qed

\begin{lemma}\label{lem511}
We have
$$\sum_{\substack{n\in\mathcal{A}(x)\\ n\not\in\mathcal{G}(p)}} w^*(p, n, \mathcal{L}_p)\leq \sum_{\substack{n\in\mathcal{A}(x)\\ n\not\in\mathcal{G}(p)}} w_n(\mathcal{L}_p).  $$
\end{lemma}

\begin{definition}\label{defn512}
Let $(h_1, \ldots, h_r)$ be an admissible $r$-tuple, $p\in(x/2, x)$. For $n\in\mathbb{Z}$, $1\leq i, l\leq r$ let 
\begin{align*}
\tilde{n}=\tilde{n}(n, i, l, p)&= n+(h_i-h_l)p\\
\mathcal{A}(i, l, p)&=\sum_{n\::\: \tilde{n}\not\in \mathcal{G}} w_n(\mathcal{L}_p).
\end{align*}
Let 
\begin{align*}
\sum (i, l, p)&:=\sum_{n\in\mathcal{A}(x)} w_n(\mathcal{L}_p)(r(\tilde{n}, u)-r^*(u))^2\\
\sum (i, l, p, j)&:=\sum_{n\in\mathcal{A}(x)} w_n(\mathcal{L}_p)r(\tilde{n}, u)^j\ \ (j\in \mathbb{N}_0).
\end{align*}
\end{definition}
\begin{lemma}\label{lem513}
$$\sum_{\substack{n\in\mathcal{A}(x)\\ n\not\in\mathcal{G}(p)}}  w_n(\mathcal{L}_p)=\sum_{1\leq i, l\leq r}\mathcal{A}(i, l, p).$$
\end{lemma}
\begin{proof}
This follows immediately from Definition \ref{defn42} and \ref{defn512}.
\end{proof}
\begin{lemma}\label{lem514}
Let $\mathcal{A}$, $w_n$ be as in (5.3), $\mathcal{L}_p$ as in Definition \ref{defn510}. Let $j\in\{1, 2\}$. Then
$$\sum(i, l, p, j)=\left(1+O\left(\frac{1}{(\log x)^{1/10}}\right)\right)\: \frac{B^g}{\phi(B)^g}\:\mathfrak{G}_B(\mathcal{L}_p)\#\mathcal{A}(x)(\log R)^g I_g(F) r^*(u)^j.$$
\end{lemma}
\begin{proof}
We only give the proof for the hardest case $j=2$ and shortly indicate the proof for $j=1$.
\begin{align*}
&\sum(i, l, p, 2)=\sum_{n\in\mathcal{A}(x)} w_n(\mathcal{L}_p) r(\tilde{n}, u)^2\\
&=\sum_{n\in\mathcal{A}(x)} w_n(\mathcal{L}_p) \left(\frac{1}{d(u)^2}\sum_{\substack{s_1\in S_u, c_{s_1}\in\{0, 1, \ldots, s_1-2\}\\ n\equiv 1-(c_{s_1}+1)^K(\bmod s_1)}} s_1^{-1}\right)
 \left(\sum_{\substack{s_2\in S_u, c_{s_2}\in\{0, 1, \ldots, s_2-2\}\\ n\equiv 1-(c_{s_2}+1)^K(\bmod s_2)}} s_2^{-1}\right)\\
&=\frac{1}{d(u)^2}\sum_{s_1, s_2\in S_u} s_1^{-1} s_2^{-1}\sum_{c_{s_1}=1}^{s_1-2}\sum_{c_{s_2}=1}^{s_2-2}\sum_{\substack{n\equiv 1-(c_{s_1}+1)^K+(h_l-h_1)p (\bmod s_1)\\ n\equiv 1-(c_{s_2}+1)^K+(h_l-h_2)p (\bmod s_2)}} w_n(\mathcal{L}_p).
\end{align*}
\end{proof}

In the inner sum we only deal with the case $s_1\neq s_2$, the case $s_1=s_2$ is giving a negligible contribution. The inner sum is non-empty if and only if the system
\begin{equation*}
    \begin{cases}
      & \tilde{n}\equiv 1-(c_{s_1}+1)^K (\bmod s_1)\\
     & \tilde{n}\equiv 1-(c_{s_2}+1)^K (\bmod s_2)
    \end{cases}\tag{*}
\end{equation*}
is solvable. In this case (*) is equivalent to a single congruence 
$${n}\equiv c+(h_l-h_i)p (\bmod s_1s_2),$$
where $e=e(s_1, s_2, c_1, c_2)$ is uniquely determined by the system (*) and $$0\leq e\leq s_1s_2-1.$$
We apply Theorem \ref{thm55} with $B$ independent of $s_1, s_2$ and with 
$$\mathcal{A}=\mathcal{A}^{(s_1, s_2)}=\{n\::\: x/2<n\leq x, n \equiv e+(h_l-h_i)p (\bmod s_1s_2)\}.$$
We have
$$\#\mathcal{A}^{(s_1, s_2)}(x)=s_1^{-1}s_2^{-1}\mathcal{A}(x)+O(1)$$
and obtain
\begin{align*}
&\sum_{n\in\mathcal{A}(x)} w_n(\mathcal{L}_p) r(\tilde{n}, u)^2\\
&= \left(1+O\left(\frac{1}{(\log x)^{1/10}}\right)\right)\frac{B^g}{\phi(B)^g}\mathfrak{G}_B(\mathcal{L}_p)\#\mathcal{A}(x)(\log R) I_g(F) r^*(u)^2.   \tag{5.20}
\end{align*}
This proves the claim for $j=2$. The proof of the case $j=1$ is analogous but simpler, since there is only the single variable of summation $s_1$. \qed

\begin{lemma}
Let the conditions be as in Lemma \ref{lem514}. Then we have
$$\sum (i, l, p)\ll \frac{B^g}{\phi(B)^g}\mathfrak{G}_B(\mathcal{L}_p)\#\mathcal{A}(x)(\log R)^g I_g(F) r^*(u)^2 (\log x)^{1/8}.$$
\end{lemma}
\begin{theorem}\label{thm51818}
Let the conditions be as in the previous lemmas. For sufficiently small $\eta_0$ we have
$$\sum_{\substack{n\in \mathcal{A}(x)\\ n\not\in \mathcal{G}(p)(\mathcal{L}_p)}} w_n(\mathcal{L}_p) \ll \frac{B^g}{\phi(B)^g}\mathfrak{G}_B(\mathcal{L}_p)\#\mathcal{A}(x)(\log R)^g I_g(F)(\log x)^{-1/10}.$$
\end{theorem}
\begin{proof}
Let $1\leq i, R \leq r$. By Definition \ref{defn512} we have 
$$\tilde{n}=  n+(h_i-h_l)p\not\in\mathcal{G}$$
which yields
$$|r(\tilde{n}, u)-r^*(u)|\geq r^*(u)(\log x)^{-1/40}.$$
Thus 
\begin{align*}
&r^*(u)^2(\log x)^{-1/20}\sum_{n\in\mathcal{A}(x)\::\: n+(h_i-h_l)p\not\in\mathcal{G}(p)} w_n(\mathcal{L}_p)\\
&\leq \sum_{n\in\mathcal{A}(x)} w_n(\mathcal{L}_p)(r(\tilde{n},u)-r^*(u))^2\\
&\leq \frac{B^g}{\phi(B)^g}\mathfrak{G}_B(\mathcal{L}_p)\#\mathcal{A}(x)(\log B)^g I_g(F)(\log x)^{-1/20} r^*(u)^2
\end{align*}
and therefore
$$\sum_{n\in\mathcal{A}(x)\::\: n+(h_i-h_l)p\not\in\mathcal{G}(p)} w_n(\mathcal{L}_p) \leq \frac{B^g}{\phi(B)^g}\mathfrak{G}_B(\mathcal{L}_p)\#\mathcal{A}(x)(\log R)^g I_g(F)(\log x)^{-1/20}.$$
The claim of Theorem \ref{thm51818} follows by summation over all pairs $(i, l)$, if $\eta_0$ is sufficiently small.
\end{proof}

We now investigate the sum (4.68)  of Theorem \ref{thm62}.

\begin{definition}\label{defn519519}
Let $x/2<p\leq x$, $L(n)=n+h_fp$. Let $L\in \mathcal{L}_p$: $1\leq i,l,\leq r$. Then we define
\begin{align*}
C(i,l,p)&:= \sum_{n\in \mathcal{A}(x)\::\:\tilde{n}\not\in\mathcal{G}} 1_p(L(n)) w_n(\mathcal{L}_p)\\
\Omega(i, l, p)&:=\sum_{n\in \mathcal{A}(x)} 1_p(L(n)) w_n(\mathcal{L}_p)(r(\tilde{n}, u)-r^*(u))\\
\Omega(i, l, p, j)&:=\sum_{n\in \mathcal{A}(x)} 1_p(L(n)) w_n(\mathcal{L}_p)r(\tilde{n}, u)^j\\
\end{align*}
\end{definition}

\begin{lemma}\label{lem520520}
Let $L, i, l, r, p$ be as in Definition \ref{defn512}. Let $j\in\{1, 2\}$. Then we have
\begin{align*}
\Omega(i, l, p, j)&:=\frac{B^{g-1}}{\phi(B)^{g-1}}\mathfrak{G}_B(\mathcal{L}_p)\#\mathcal{P}_{L, \mathcal{A}}(x)(\log R)^{g+1}\\
&\ \ \times J_g(F)r^*(u)^j (1+((\log x)^{-1/10}))\\
&\ \ +O\left(\frac{B^{g}}{\phi(B)^{g}}\mathfrak{G}_B(\mathcal{L}_p)\#\mathcal{A}(x)(\log R)^{g-1} J_g(F)r^*(u)^j\right).
\end{align*}
\end{lemma}
\begin{proof}
We only give the proof for the hardest case $j=2$. The case $j=1$ is analogous but simpler. We have 
\begin{align*}
\Omega(i, l, p, 2)&=\sum_{n\in \mathcal{A}(x)} 1_p(L(n))w_n(\mathcal{L}_p) r(\tilde{n}, u)^2\\
&=\sum_{n\in \mathcal{A}(x)} 1_p(L(n))w_n(\mathcal{L}_p) \left(\frac{1}{d(u)^2}\sum_{\substack{s_1\in S_u, c_{s_1}\in\{1, \ldots, s_1-2\}\\ \tilde{n}\equiv 1-(c_{s_1}+1)^K(\bmod s_1)}} s_1^{-1}\right)\\
&\ \ \times\left(\sum_{s_2\in S_u, c_{s_2}\in\{1, \ldots, s_2-2\}} s_2^{-1}\right)\\
&=  \frac{1}{d(u)^2} \sum_{s_1, s_2\in S_u} s_1^{-1} s_2^{-1} \sum_{c_{s_1}=1}^{s_1-2}\sum_{c_{s_2}=1}^{s_2-2}\sum_{\substack{\tilde{n}\equiv 1-(c_{s_1}+1)^K+(h_l-h_i)p(\bmod s_1)\\ \tilde{n}\equiv 1-(c_{s_2}+1)^K+(h_l-h_i)p(\bmod s_2)}} 1.
\end{align*}

We deal only with the case $s_1\neq s_2$ for the inner sum, the case $s_1=s_2$ giving a negligible contribution. The inner sum is non-empty if and only if the system
\begin{equation*}
    \begin{cases}
      & {n}\equiv 1-(c_{s_1}+1)^K (\bmod s_1)\\
     & {n}\equiv 1-(c_{s_2}+1)^K (\bmod s_2)
    \end{cases}\tag{*}
\end{equation*}
is solvable.\\
In this case the system is equivalent to a single congruence \mbox{$n\equiv e(s_1, s_2, c_1, c_2)$} is uniquely determined by the system (*) and \mbox{$0\leq e\leq s_1s_2-1$.} The inner sum then takes the form
$$\sum_{\substack{n\equiv c(\bmod s_1s_2)\\ n\in\mathcal{A}(x)}} 1_p(L(n))w_n(\mathcal{L}_p).$$
By the substitution $n=ms+c$, we get
$$L(n)=L^*(m,s)=ms+e+h_fp.$$
We set $\mathcal{L}_p=\{L_{h_i}\}$, where $L_{h_i}(n)=n+h_ip$ is replaced by the set $\mathcal{L}_{p,s}=\{L_{h_i, s}\}$, where
$$L_{h_i, s}(m)=ms+e+(h_i+h_f)p.$$
We thus have 
\begin{align*}
\sum (s)&:=\sum_{\substack{n\equiv e(\bmod s)\\ n\in\mathcal{A}(x)}} 1_p(L(n)) w_n(\mathcal{L}_p)\\
&=\sum_{m\in \mathcal{A}\left(\frac{x}{s}\right)} w_m(\mathcal{L}_{p,s}) 1_p(L^*(m,s))+O(1).
\end{align*}

We apply Theorem \ref{thm59} with $\mathcal{A}=\mathbb{N}$, $x/s$ instead of $x$, $L(\cdot)=L^*(\cdot, s)$, $\mathcal{L}=\mathcal{L}_{p,s}$. We have
$$\mathfrak{G}_B(\mathcal{L}_p)=\mathfrak{G}_B(\mathcal{L}_{p,s})\left(1+O\left(\frac{1}{\log x}\right)\right).$$
From Bombieri's Theorem it can easily be seen that the conditions (5.5) are satisfied for all $s$ with the possible exception of $s\in\mathcal{E}$, $\mathcal{E}$ being an exceptional set, satisfying
$$\sum_{s\in\mathcal{E}} s^{-1} \ll (\log x)^{-4}.$$
For $s\in \mathcal{E}$ we use the trivial bound $1_p(L^*(m,s))=O(1)$. Thus, we obtain the claim of Lemma \ref{lem520520} for the case $j=2$.\\
The proof for $j=1$ is analogous but simpler, since we have only to sum over the single variable $s_1$.
\end{proof}

\begin{lemma}
Let $i, l, p$ as in Definition \ref{defn512}. We have
\begin{align*}
\Omega(i,l,p)&=O\left(\frac{B^{g-1}}{\phi(B)^{g-1}}|\mathfrak{G}_B(\mathcal{L}_p)|\# P_{L, \mathcal{A}}(x) (\log R)^{g+1} J_g(F) r^*(u)^2 (\log x)^{-1/10}\right)\\
&\ \ +O\left(\frac{B^{g}}{\phi(B)^{g}}|\mathfrak{G}_B(\mathcal{L}_p)|\# \mathcal{A}(x)(\log R)^{g-1} J_g(F)r^*(u)^2\right).
\end{align*}
\end{lemma}
\begin{proof}
By Definition \ref{defn519519}, we have
$$\Omega(i,l,p)=\Omega(i,l,2)-2r^*(u)\:\Omega(i,l,p,2)+r^*(u)^2\:\Omega(i,l,p,0).$$
\end{proof}

\begin{theorem}\label{thm522522}
Let $p, L(n)$ be as in Definition \ref{defn519519}. Then we have
$$\sum_{\substack{n\in\mathcal{A}(x)\\ n\not\in \mathcal{G}(p)}} 1_p(L(n)) w_n(\mathcal{L}_p)\ll \frac{B^{g-1}}{\phi(B)^{g-1}}\:\mathfrak{G}_B(\mathcal{L}_p)\# P_{L, \mathcal{A}}(x)(\log R)^{g+1} J_g(F)(\log x)^{-1/10}.$$
\end{theorem}
\begin{proof}
Let $1\leq i, l\leq r$. By Definition \ref{defn512} we have
$$\tilde{n}=n+(h_i-h_l)p\in\mathcal{G}.$$
It follows that 
$$|r(\tilde{n}, u)-r^*(u)|\geq r^*(u)(\log x)^{-1/40}.$$
Thus
\begin{align*}
&r^*(u)^2(\log x)^{-1/20} \sum_{n\in\mathcal{A}(x)\::\: n\not\in \mathcal{G}(p)} 1_p(L(n)) w_n(\mathcal{L}_p)\\
&\ll \frac{B^{g-1}}{\phi(B)^{g-1}}\:\mathfrak{G}_B(\mathcal{L}_p)\# P_{L, \mathcal{A}}(x)(\log R)^{g+1} J_g(F) r^*(u)^2(\log x)^{-1/10}\\
&+\frac{B^{g}}{\phi(B)^{g}}\:\mathfrak{G}_B(\mathcal{L}_p)\# D(x)(\log R)^{g-1} J_g(F) r^*(u)^2.
\end{align*}
The second term is absorbed in the first one, since by Definition:
$$x^{\theta/10}\leq R\leq x^{\theta/3}$$
and thus 
$$\log R \asymp \log x.$$
Therefore 
$$\mathcal{L}(i, l, p)\ll  \frac{B^{g-1}}{\phi(B)^{g-1}} \:\mathfrak{G}_B(\mathcal{L}_p)\# P_{L, \mathcal{A}}(x)(\log R)^{g+1} J_g(F) (\log x)^{-1/20}.$$
The claim of the Theorem \ref{thm522522} now follows by summing over all pairs $(i, j)$. 
\end{proof}
We now can conclude the proof of Theorem \ref{thm62} and therefore also the proof of Theorem 1.1.\\
By Theorems \ref{thm57}, \ref{thm59}, \ref{thm51818} and \ref{thm522522} we have 
\[
\sum_{n\in\mathcal{A}(x)}  w_{(K)}(p, n)=\left(1+O\left(\frac{1}{(\log x)^{1/100}}\right)\right) \sum_{n\in\mathcal{A}(x)}  w_n(\mathcal{L}_n)  \tag{5.21}
\]
and 
\[
\sum_{n\in\mathcal{A}(x)} 1_p(L(n))  w_{(K)}(p, n)=\left(1+O\left(\frac{1}{(\log x)^{1/100}}\right)\right) \sum_{n\in\mathcal{A}(x)} 1_p(L(n))  w_n(\mathcal{L}_n).  \tag{5.22}
\]

The deduction of the equation (4.66) and (4.67) of Theorem \ref{thm62} can thus be deduced from results on the sums on the right hand side of equation (5.21) and (5.22).


\section{The $K$-version of Large Gap results for primes from special sequences}

In joint work with Maier \cite{MR_Beatty}, \cite{MR_PS}, the author of this paper established the $K$-version for special sequences of primes: Beatty primes and Piatetski-Shapiro primes.\\
We recall the following definitions:
\begin{definition}\label{defn61}
For two fixed real numbers $\alpha$, $\beta$ the corresponding non-homogeneous Beatty sequence is the sequence of integers defined by
$$B_{\alpha,\beta}:=([\alpha n+\beta])_{n=1}^\infty.$$
\end{definition}
\begin{definition}\label{defn62}
For an irrational number $\gamma$ we define its type $\tau$ by the relation
$$\tau:=\sup\{\rho\in\mathbb{R}\::\: \lim\inf n^\rho \| \gamma n\| =0\}.$$
\end{definition}
\begin{definition}\label{defn63}
Let $c>1$ be a fixed constant. A prime of the form $[l^c]$ is called Piatetski-Shapiro prime.
\end{definition}


In the paper \cite{MR_Beatty}, the following Theorem is proved:
\begin{theorem} (Theorem 1.3 of \cite{maynard2})\\
Let $k\geq 2$ be an integer. Let $\alpha, \beta$ be fixed real numbers with $\alpha$ being a positive irrational and of finite type. Then there is a constant $C>0$, depending only on $\alpha$ and $\beta$, such that for infinitely many $n$ we have:
$$p_{n+1}-p_n\geq C\: \frac{\log p_n\:\log_2 p_n\:\log_4 p_n}{\log_3 p_n}$$
and the interval $[p_n, p_{n+1}]$ contains the $K$-th power of a prime $\tilde{p}\in{B}_{\alpha,\beta}$.
\end{theorem}

In the paper \cite{baker_b}, the following Theorem is proved:
\begin{theorem}  (Theorem 2.1 of \cite{baker_b})\\
Let $c\in(1,18/17)$ be fixed, $K\in\mathbb{N}$, $K\geq 2$. Then there is a constant $C>0$, depending only on $K$ and $C$, such that for infinitely many $n$ we have
$$p_{n+1}-p_n\geq C\: \frac{\log p_n\: \log_2 p_n\: \log_4 p_n}{\log_3 p_n}$$
and the interval $[p_n, p_{n+1}]$ contains the $K$-th power of a prime $\tilde{p}=[l^c]$.
\end{theorem}
We now give a short sketch of the proof of these theorems.\\
These proofs are modifications of the proofs of the $K$-versions of the Large Gap result in Sections 4 and 5. One applies the matrix method.\\
The matrices $\mathcal{M}$ are defined in a manner similar to their definition in the deduction of Theorem 2.1 of \cite{baker_b}. Once again choose $x$, such that $P(C_0x)$ is a good modulus.\\
The only major modification is that one does not count primes of the form 
$$a_{r,1}=(m_0+1+r) P(x)$$
in the first column ((1) of the matrix $\mathcal{M}$ but only such primes from $B_{\alpha, \beta}$ (Beatty primes) and from $P^{(c)}=\{[l^c]\ \text{prime}\}.$\\
For the count of Beatty resp., Piatetski-Shapiro primes in the column C(1):

\begin{lemma} (Lemma 3.1 of \cite{MR_Beatty})\\
Let $\alpha$ and $\beta$ be fixed real numbers with $\alpha$ a positive irrational and of finite type. Then there is a constant $\kappa>0$, such that for all integers
\[
0\leq a<q\leq N^\kappa \tag{3.1}
\]
with $(a,q)=1$, we have
$$\sum_{\substack{n\leq N \\ [ \alpha n+\beta] \equiv a\:(\bmod\: q)}} \Lambda([ \alpha n+\beta])=\alpha^{-1} \sum_{\substack{m\leq [ \alpha N+\beta] \\ m\equiv a\:(\bmod\: q)}} \Lambda(m)+O(N^{1-\kappa})\:,$$
where the implied constant depends only on $\alpha$ and $\beta$.
\end{lemma}

\begin{theorem} (Theorem 8 of \cite{baker_b})\\
Let $a$ and $d$ be coprime integers, $d\geq 1$. For fixed $c_0\in(1, 18/17)$ we have (with $\gamma=1/c_0$):
\begin{align*}
\pi_{c_0}(w; d, a)=&\gamma w^{\gamma-1}\pi(w; d, a)\\
&\ +\gamma(1-\gamma)\int_2^w  u^{\gamma-2} \pi(u; d, a)  du+O\left(w^{17/39+7\gamma/13+\epsilon}\right).
\end{align*}
$$(\pi_{c_0}(w; d, a)=\#\{p\in \mathcal{P}^{(c_0)}\::\: p\leq w, p\equiv a \bmod d).$$
\end{theorem}

\noindent\textbf{Acknowledgments.}\\ 
The author wishes to express his gratitude to H. Maier for  extensive discussions and close communication during the preparation of this paper. His support has been invaluable. The author wishes to also thank the anonymous referees for reading the manuscript in detail and for providing very constructive comments which helped improve the presentation of this work.

\end{document}